\documentclass[a4paper, 11 pt, twoside]{amsart}

\usepackage{srollens-en}[2011/03/03]

\usepackage[left = 3.5cm, right = 3.5cm, headsep = 6mm,
footskip = 10mm, top = 35mm, bottom = 35mm, footnotesep=5mm, headheight =
2cm]{geometry}

\usepackage[usenames, dvipsnames]{xcolor}

\usepackage{tikz, tikz-cd}
\usetikzlibrary{positioning, intersections}
\usetikzlibrary{decorations.pathreplacing,decorations.markings}
\tikzset{
  mid arrow/.style={postaction={decorate,decoration={
        markings,
        mark=at position .5 with {\arrow[#1]{stealth}}
      }}},
}
\tikzset{
C1col/.style = { MidnightBlue, thick},
C1/.style ={C1col, thick, postaction={mid arrow}},
C2col/.style = {Orange, thick },
C2/.style ={C2col,thick, postaction={mid arrow}},
C12col/.style = {ForestGreen, thick},
C12/.style ={C12col,thick, postaction={mid arrow}},
CC/.style = {thick, C2col, postaction={mid arrow}},
L1/.style = {MidnightBlue, thick, postaction={mid arrow}},
L2/.style = {ProcessBlue,thick, postaction={mid arrow}},
L12/.style = {NavyBlue, thick, postaction={mid arrow}},
L3col/.style = {LimeGreen, thick},
L3/.style = {L3col, postaction={mid arrow}},
L4/.style = {OliveGreen, thick, postaction={mid arrow}},
L34/.style = {ForestGreen, thick, postaction={mid arrow}},
Qs/.style ={black!50!white, opacity=1, every node/.style={black, opacity=1, font  = \scriptsize}},
basepoint/.style ={black, opacity = 1, every node/.style={black}}
}

\usepackage{booktabs, multirow}
\usepackage{enumitem}
\setlist[description]{leftmargin=0cm,  labelindent=\parindent}

\usepackage[unicode,bookmarks, pdftex, backref = page]{hyperref}
\hypersetup{colorlinks=true,citecolor=NavyBlue,linkcolor=NavyBlue,urlcolor=Orange, pdfpagemode=UseNone, breaklinks=true}

\renewcommand{\tilde}{\widetilde}
\newcommand{\pp}{\mathbb P}
\newcommand{\OO}{\mathcal O}
\newcommand{\wt}{\widetilde}
\newcommand{\epsi}{\epsilon}
\newcommand{\I}{\mathrm{i}}

\newcommand{\caseR}{{$(R)$}}
\newcommand{\caseB}[1]{$(B_{#1})$}
\newcommand{\caseE}[1]{$(E_{#1})$}
\newcommand{\caseP}[1]{$(P_{#1})$}
\newcommand{\casedP}{$(dP)$}

\title{Gorenstein stable Godeaux surfaces}

\author{Marco Franciosi}
\address{Marco Franciosi\\Dipartimento di Matematica\\Universit\`a di Pisa \\Largo B. Pontecorvo 5\\I-56127  Pisa\\Italy}
\email{marco.franciosi@unipi.it}

\author{Rita Pardini}
\address{Rita Pardini\\Dipartimento di Matematica\\Universit\`a di Pisa \\Largo B. Pontecorvo 5\\I-56127  Pisa\\Italy}
\email{rita.pardini@unipi.it}

\author{S\"onke Rollenske}
\address{S\"onke Rollenske\\FB 12/Mathematik und Informatik\\
Philipps-Universit\"at Marburg\\
Hans-Meerwein-Str. 6\\
35032 Marburg\\
Germany}
\email{rollenske@mathematik.uni-marburg.de}

\begin{document}
\begin{abstract}
We classify Gorenstein stable numerical Godeaux surfaces with worse than canonical singularities and compute their fundamental groups.
\end{abstract}
\subjclass[2010]{14J10, 14J25, 14J29}
\keywords{stable surface, Godeaux surface}
\maketitle

\setcounter{tocdepth}{2}
\tableofcontents
\section{Introduction}

Surfaces of general type with the smallest possible invariants, namely $K_X^2=1$ and $p_g(X) = q(X) =0$ are called (numerical) Godeaux surfaces.
For  a Godeaux surface $X$   with  at most canonical singularities, the algebraic fundamental group $\pi_1^{\alg}(X)$ is known to be cyclic of order $d\le 5$
 (see \cite{miyaoka76, reid78}). Hence $\pi_1^{\alg}(X)$ and the torsion subgroup $T(X)$ of $\Pic(X)$ are finite abelian groups dual to each other. 
  The cases $d=3,4,5$ are completely understood   \cite{reid78, coughlan-urzua16}, and quite a few  examples of Godeaux surfaces with torsion $\IZ/2$ or trivial have been constructed since then\footnote{See e.\,g.\ \cite{Barlow84, Barlow85, catanese-debarre89, Inoue94, Werner94, Dol-Wer99, Dol-Wer01, lee-park07, rana-tevelev-urzua15, coughlan16}.}.
  
In this paper we classify  non-classical Gorenstein stable Godeaux surfaces, i.e.,  Gorenstein stable surface $X$ with $K^2_X=1$ and $p_g(X)=q(X)=0$ which have worse than canonical singularities, and calculate their (topological) fundamental groups.

Denote by $\gothM_{1,1}$ the Gieseker moduli space of \emph{classical} Godeaux surfaces and by $\overline \gothM_{1,1}$ its compactification, the moduli space of stable Godeaux surface (see   \cite{kollar12, KollarModuli} and references therein). The space $\gothM_{1,1}$ is conjectured to have exactly $5$ irreducible components, one for each possible $\pi_1^{\alg}(X)$. 

Some of the examples we construct turn out to be smoothable, that is, they are in the closure of $\gothM_{1,1}$, some are known not to be smoothable and they make up an irreducible component of  $\overline\gothM_{1,1}$. The smoothability of the examples with small fundamental group is still to be investigated. So $\overline\gothM_{1,1}$ has at least six irreducible components but might still turn out to be connected. 

Instead of giving a detailed description of the single cases, we collect an overview of our results in Table  \ref{tab: list}. The statements made in the table and some further applications will be stated and proved in Section \ref{sect: conclusions}; the explicit constructions can be found in the sections indicated in the table.

\begin{table}\caption{List of non-classical Gorenstein stable Godeaux surfaces}\label{tab: list}
\begin{tabular}{clccl}
 \toprule
 $|\pi_1(X)| $ & case & normal & smoothable & reference\\
 \midrule
   & \caseR & \checkmark & \checkmark &  Section \ref{sect: normal classification}\\
 {$5$} & \caseE5 general & --- & \checkmark & Section \ref{sect: case E} \\
  & $X_{1.5}$ &--- & \checkmark & Section \ref{sect: case P} \\
 \midrule
  \multirow{4}*{$4$}  & \caseR & \checkmark & \checkmark &  Section \ref{sect: normal classification}\\
& \caseB1 & \checkmark & \checkmark &  Section \ref{sect: normal classification} \\
   & \caseP1 & --- & \checkmark & Section \ref{sect: case P} \\
   & \caseE4 general& --- & \checkmark & Section \ref{sect: case E} \\
  \midrule
  \multirow{4}*{$3$}& \caseR & \checkmark & \checkmark &  Section \ref{sect: normal classification}\\
  & \caseB2& \checkmark & \checkmark &  Section \ref{sect: normal classification} \\
   & \caseP3 & --- & \checkmark & Section \ref{sect: case P} \\
   & \caseE3 general & --- & \checkmark & Section \ref{sect: case E} \\
     \midrule
   $2$ & \caseE2 general& --- & unknown & Section \ref{sect: case E} \\
     \midrule
   & \caseP2 & --- & unknown & Section \ref{sect: case P} \\
   & \caseE1  & --- & unknown & Section \ref{sect: case E} \\
  $1$ & \caseE{2}, reducible polarisation&--- & unknown & Section \ref{sect: case E}\\
   & \caseE{m}, reducible pol., $m\geq3$&--- &  \checkmark\  (Rem.\ \ref{rem: smoothing reducible polarisation}) & Section \ref{sect: case E}\\
    & \casedP & --- & no  \cite{rollenske16} & Section \ref{sect: case dP}\\

 \bottomrule
\end{tabular}
\end{table}

\subsection*{Acknowledgements}  The first and the second  author  are  members of GNSAGA of INDAM. The third author is grateful for support of the DFG through the Emmy Noether program and partially through SFB 701. 
This project was partially supported by PRIN 2010 ``Geometria delle Variet\`a Algebriche'' of italian MIUR. 

We are indepted to Angelo Vistoli for Proposition \ref{prop:angelo}. 
The third author would like to thank Stephen Coughlan and  Roberto Pignatelli for discussions about Godeaux surfaces.

\section{Normal non-canonical stable Godeaux surfaces}\label{sect: normal}

\subsection{Classification of possible cases}\label{sect: normal classification}
In this section we study  normal  Gorenstein  stable  Godeaux surfaces that are ``non-canonical'', namely have worse than canonical singularities.
We first refine the results  in  \cite[\S 4]{FPR15a} in this particular case. 

\begin{prop}\label{prop: normal-class} Let $X$ be a normal Gorenstein stable  surface with $\chi(X) = K^2_X=1$  and  let $\epsi\colon \wt X\to X$ be the  minimal desingularisation. 

Then $X$ has precisely one elliptic singularity and the exceptional divisor $D$ of $\epsi$ is a smooth elliptic curve; in addition, one of the following occurs: 

\begin{itemize}
 \item[(B)] There exists a bi-elliptic surface $X_{\min}$,  an irreducible   divisor  $D_{\min}$ on $X_{\min}$ and a point $P$ such that:
\begin{itemize}
\item $D^2_{\min}=2$ and $P\in D_{\min}$ has multiplicity $2$
\item  $\wt X$ is the blow-up of $X_{\min}$ at $P$
\item $X$ is obtained from $\wt X$ by blowing down the strict transform  of $D_{\min}$.
\end{itemize}
\item [\caseR] $\wt X$ is a ruled surface with $\chi(\wt X)=0$.
\end{itemize}
\end{prop}
We first prove:
\begin{lem}\label{lem:godeaux-normal-1}
In the assumptions of Proposition \ref{prop: normal-class}, one has:
\begin{enumerate}
\item $q(\wt X)=1$ and $p_g(\wt X)=0$
\item $X$ has exactly one elliptic singularity and the corresponding exceptional divisor is a smooth elliptic curve
\end{enumerate} 
\end{lem} 
\begin{proof}
 By \cite[Thm.~4.1]{FPR15a}, the singularities of  $X$ are either  canonical or elliptic Gorenstein. So we write $\epsilon^* K_X=K_{\wt X}+\wt D_1+\dots +\wt D_{k}$, with the $\wt D_i$ disjoint $2$-connected effective  divisors with $p_a(\wt D_i)=1$. We write  $\wt D=\wt D_1+\dots +\wt D_{k}$.  
 
 Since $k>0$ by assumption, by  \cite[Lem.~4.1]{FPR15a} we have $\chi(\wt X)=\chi(X)-k=1-k\le 0$, hence $q(\wt X)>0$. Again by \cite[Lem.~4.1]{FPR15a} we  also have $p_g(\wt X)\le h^0(K_{\wt X}+\wt D)=p_g(X)=0$;   it follows that the Albanese image of $\wt X$ is a curve $B$.  
  For every $i=1, \dots, k$ the standard restriction sequence induces  an injection $\IC\simeq H^0(K_{\wt D_i}) \to H^1(K_{\wt X})$, hence no  $\wt D_i$ is contained in a fibre of the Albanese map of $\wt X$.  It follows that $B$ has genus 1, and therefore   $q(\wt X)=1$, $\chi(\wt X)=0$, $k=1$ and $\wt D=\wt D_1$ is a smooth elliptic curve. 
  \end{proof}
  
 \begin{proof}[Proof of Proposition \ref{prop: normal-class}]
 By Lemma \ref{lem:godeaux-normal-1},  by \cite[Thm.~4.1]{FPR15a} and by the classification of surfaces, it is enough to exclude that $\wt X$ is a minimal properly elliptic surface and that the exceptional divisor $\wt D$ of the minimal resolution $\epsilon \colon \wt X\to X$ is a smooth curve of genus 1 with  $\wt D^2=-1$.
 
 So assume by contradiction that this is the case. Since $c_2(\wt X)=0$ by Noether's formula,  applying the  formula for computing the Euler characteristic  (\cite[Prop.~III.11.4]{BHPV})  to the elliptic  fibration $f\colon \wt X\to B$ one sees that all fibres of $f$ have smooth support, namely $f$ is a \textit{quasi-bundle} (cf. \cite{serrano93}, \cite{serrano96}).
 So the surface  $\wt X$ is a free  quotient $(F\times C)/G$ where:
 \begin{itemize} 
 \item $F$ is a curve  of genus 1, $C$ is a curve of genus $g>1$
 \item $G$ is a finite group that acts  faithfully on $F$ and  $C$; we let $G$ act   diagonally on $F\times C$
 \item the elliptic fibration is induced by the second projection $p_2\colon F\times C \to C$.
 \end{itemize}
 Let now $\Gamma\subset F\times C$ be an irreducible component of the preimage of  $\wt D$: the map $\Gamma\to \wt D$ is \'etale, and so $\Gamma$ is an elliptic curve and is mapped to a point by $p_2$. It follows that $\wt D$ is mapped to a point by $f$, hence $\wt D^2=0$, against the assumptions. 
 \end{proof}
  
The surfaces of type $(B)$ can be described more explicitly. We start by giving two examples.
  \begin{exam}\label{ex:godeaux-bi-elliptic}
 By Proposition \ref{prop: normal-class}, a normal Gorenstein surface of type (B) is determined by a pair   $(X_{\min},  D_{\min})$ where $X_{\min}$ is a minimal bi-elliptic surface and $D_{\min}$ is  an irreducible  divisor with $D_{\min}^2=2$ that has a double point $P$. 
 We describe two types of such pairs:
\begin{itemize}
\item[\caseB{1}] 
Let $E$ be an elliptic curve, set $A:=E\times E$  and let  generators  $e_1,e_2\in G\cong (\IZ/2)^2$ act on $A$ as follows:
$$(x,y)\stackrel{e_1}{\mapsto}(x+\tau_1,y+\tau_1); \quad (x,y)\stackrel{e_2}{\mapsto}(x+\tau_2,-y),$$
where $\tau_1$ and $\tau_2$ generate $E[2]$.
The group $G$ acts freely on $A$ and the surface  $X_{\min}:=A/G$ is bi-elliptic. 

The diagonal $\Delta\subset A$ is $e_1$-stable and it intersects $e_2\Delta$ transversally at 4 points. The divisor $\Delta +e_2\Delta$ is $G$-invariant and its image $D_{\min}$ in $X_{\min}$ is and irreducible curve with $D_{\min}^2=2$ that has  a node $P$.

\item[\caseB{2}]Let $\zeta:=e^{\frac{2\pi i}{3}}$,  denote by $E$ the elliptic curve $\IC/(\IZ+\IZ\zeta)$ and denote by $\rho$ the order 3 automorphism of $E$ induced by multiplication by $\zeta$. We  let $A:=E\times E$  and consider the following automorphisms of $A$:
$$(x,y)\stackrel{e_1}{\mapsto}(x+\tau_1,y+\tau_1); \quad (x,y)\stackrel{e_2}{\mapsto}(x+\tau_2,\rho y),$$
where $\tau_1=\frac{1-\zeta}{3}$ and $\tau_1$ and $\tau_2$ generate $E[3]$. 
Since $\rho(\tau_1)=\tau_1$, the automorphisms  $e_1$ and $e_2$ commute,  and they generate a  subgroup $G$ of $\Aut(A)$ isomorphic to $(\IZ/3)^2$. The group $G$  acts freely on $A$ and the surface  $X_{\min}:=A/G$ is bi-elliptic. 

 The diagonal $\Delta\subset A$ is $e_1$-stable 
  and it intersects $e_2\Delta$ and $(2e_2)\Delta$ transversally at 3 points. By symmetry, $e_2\Delta$ and $(2e_2)\Delta$ also intersect  transversally at 3 points. So the divisor $Z:=\Delta +e_2\Delta+(2e_2)\Delta$ is $G$-invariant and $Z^2=18$. It follows that the image $D_{\min}$ of $Z$ in $X$ is 
   irreducible with $D_{\min}^2=2$.  The induced map $\Delta/\langle e_1\rangle\to D_{\min}$ is birational, hence the normalisation of $D_{\min}$ has genus 1 and by the adjunction formula $D_{\min}$ has precisely one double point. 
   
  Note that the group $\Aut(E)$ acts transitively on the complement of the subgroup $\langle\tau_1\rangle$ in $E[3]$, so that the construction is independent of the choice of $\tau_2$.
  \end{itemize}
\end{exam}

  \begin{prop}\label{prop:godeaux-normal-B}
Let $X$ be a normal  Gorenstein  Godeaux surface of type (B). Then $X$ is obtained as in case \caseB{1} or \caseB{2} of Example \ref{ex:godeaux-bi-elliptic}.
\end{prop}
\begin{proof}
Let $(X_{\min}, D_{\min})$ be the pair giving the surface $X$ as in Proposition \ref{prop: normal-class}. 
The surface $X_{\min}$, being bi-elliptic,  is  a quotient $(B\times C)/G$, where $B$ and $C$ are elliptic curves, $G$ is a finite abelian group of order $m$ that  acts faithfully on $B$  and  $C$ and:
\begin{itemize}
\item the action of $G$  on  $B\times C$ is the diagonal action
\item  $B/G$ is elliptic and $C/G$ is rational. 
\end{itemize} 

Denote by $\Gamma$ the preimage of $D_{\min}$ in $B\times C$ and write $\Gamma=\Gamma_1+\dots +\Gamma_r$, with the $\Gamma_i$ irreducible.  
Notice that by construction the $\Gamma_i$ form a $G$-orbit and that   $r=[G:H]$, where $H$ is the stabilizer of  $\Gamma_1$.  We claim that all elements of $H$ are translations. Indeed, assume by contradiction that $h\in H$ is not a translation. Then, denoting by $\sigma$ a generator of $H^0(\omega_C)$,   we have $h^*\sigma=\lambda \sigma$ for some  $1\ne \lambda \in \IC$. On the other hand, the action of $H$ on $\Gamma_1$ is free, since $G$ acts freely on $A$, and therefore we have $h^*\tau=\tau$, where $\tau$ is a generator of $H^0(\omega_{\Gamma_1})$. So, denoting by $p_2\colon B\times C\to C$ the second projection, we conclude that $p_2^*\sigma$ restricts to zero on $\Gamma_1$, and therefore $\Gamma_1$ is a fibre of $p_2$. 
This implies that $D_{\min}$ is contracted by the map $X_{\min}\to C/G$, which is impossible since $D^2_{\min}=2>0$. 

The normalisation of $\Gamma$ is an \'etale cover of the exceptional divisor $\wt D$ of the minimal desingularisation $\epsilon\colon  \wt X\to X$, hence the  $\Gamma_i$  have geometric genus 1; 
  since $B\times C$ is an abelian variety, it follows that the $\Gamma_i$ are smooth elliptic  with  $\Gamma_i^2=0$.
By standard computations, $H^{1,1}(X_{\min})$  is 2-dimensional, generated by the classes of the images $\bar B$ and $\bar C$ of $B\times\{0\}$ and $\{0\}\times C$, hence $D_{\min}$ is a linear combination with rational coefficients of $\bar B$ and $\bar C$.
Pulling back to $B\times C$,  one sees that $\Gamma$ is numerically equivalent to $\beta(B\times \{0\}+\gamma(\{0\}\times C)$ for some $\beta, \gamma\in \IQ$.  We have $m=\frac 12\Gamma^2=\beta\gamma$. Moreover, since the intersection numbers  $\Gamma_i(B\times\{0\})$ and $\Gamma_i(\{0\}\times C)$ are independent of $i$, we have  $\beta=r(\Gamma_1(\{0\}\times C))$ and $\gamma=r(\Gamma_1(B\times \{0\}))$. 
Hence $m=|G|$ is divisible by $[G:H]^2$. Since $H$  consists of translations, 
 this remark rules out five of the seven cases of Bagnera-de Franchis list (cf. \cite[List~VI.20]{beauville-book}) and we are left with the following possibilities:
\begin{gather*}
  G=\IZ/2\times\IZ/2 \text{ and } r=2,\\
G=\IZ/3\times \IZ/3 \text{ and } r=3.
\end{gather*}
In both cases, by the above computations the $\Gamma_i$ project isomorphically onto $B$ and $C$, hence $B$ and $C$ are isomorphic and we may identify $B$ and $C$  in such a way that, say, $\Gamma_1$ is the diagonal.  Hence  we have either case \caseB{1} or \caseB{2}
\end{proof}

\subsection{Computation of fundamental groups}

We devote the rest of the section the description of the fundamental group and thus of  the torsion group of normal non-canonical Gorenstein stable Godeaux surfaces.
Our results, more precisely Proposition \ref{prop: pi1 for B}, Lemma \ref{lem:pi1-R}, and Proposition \ref{prop:torsion-bound-normal}, are summarised in  the following:
\begin{thm}\label{thm:godeaux-normal-torsion}
 Let $X$ be a normal non-canonical   Gorenstein  Godeaux surface. 
 Then:
 \begin{enumerate}
 \item If $X$ is of type \caseB{1} then $\pi_1(X)$ and $T(X)$ are cyclic of order 4. 
\item  If $X$ is of type \caseB{2} then $\pi_1(X)$ and $T(X)$ are cyclic of order 3. 
\item If $X$ is of type \caseR, then  $\pi_1(X)$ and $T(X)$  are cyclic of  order $3\le d\le 5$.
\end{enumerate}
\end{thm}

 We deal first with surfaces of type (B) and start with two preparatory Lemmas.
 \begin{lem}\label{lem:cover-min}
  Let $X$ be a normal  Godeaux surface of type (B), obtained from a pair $(X_{\min}, D_{\min})$ and let $G$ be a finite group of order $d$. 
 Then the connected   \'etale $G$-covers $Y\to X$  are naturally in one-to-one correspondence with the connected   \'etale $G$-covers $Y_{\min}\to X_{\min}$ such that the preimage of $D_{\min}$ in $X_{\min}$ has $d$ irreducible components. 
 \end{lem}
 \begin{proof} The correspondence goes as follows: if $Y_{\min}\to X_{\min}$ is  an \'etale cover as in the statement, its pull-back $\wt Y\to \wt X$ is an \'etale $G$-cover whose restriction to $\wt D$ is the trivial cover. If $\wt Y\to Y\to X$ is the Stein factorisation of the induced morphism $\wt Y\to X$ then $Y\to X$ is a connected \'etale $G$-cover.  
 
 The inverse correspondence is described in an analogous way: if $Y\to X$ is an \'etale $G$-cover and  $\wt Y\to \wt X$  is  the induced  \'etale cover of $\wt X$, then the exceptional curve $E$ of the blow up $\wt X \to X_{\min}$ pulls back to a   sum  $E_1+\dots +E_d$ of disjoint $(-1)$-curves of $\wt Y$. Contracting each  $E_i$ to  a point yields an \'etale cover $Y_{\min}\to X_{\min}$.
 \end{proof}

\begin{lem}\label{lem: pi1 contraction}
 Let $f\colon X\to Y$ be a birational morphism of normal surfaces and let $E_1, \dots, E_n$ be the connected components of the exceptional locus with inclusions $\iota_j\colon E_j \into X$. If $\pi_1(X)$ is abelian then \[\pi_1(Y) =\frac{\pi_1(X)}{\langle \im {\iota_j}_*\mid j=1, \dots, n \rangle}\]
\end{lem}
\begin{proof}
Fix an arbitrary base point in $X$. A priori the image of $\pi_1(E_j)$ in $\pi_1(X)$ is determined only up to conjugation, but in our situation it  is uniquely determined since $\pi_1(X)$ is abelian. The result then  follows by induction from the Seifert--van Kampen-type result proved in \cite[Cor.~3.2\,\refenum{ii}]{FPR15b}. 
\end{proof}

 \begin{prop}\label{prop: pi1 for B}Let $X$ be a normal Godeaux surface of type (B). Then:
 \begin{enumerate}
 \item If $X$ is of type \caseB{1} then $\pi_1(X)$ and $T(X)$ are cyclic of order 4. 
\item  If $X$ is of type \caseB{2} then $\pi_1(X)$ and $T(X)$ are cyclic of order 3. 
\end{enumerate}
\end{prop}
\begin{proof}
\refenum{i} We will directly construct the universal cover of $X$, starting with a degree $4$ cover of $X_{\min}$ which is trivial on the normalisation of $D_{\min}$, as suggested by Lemma \ref{lem:cover-min}. We continue to use the notation introduced in Example \ref{ex:godeaux-bi-elliptic} and let $E_1:=E/\langle \tau_1\rangle$. Fix in addition an element $\eta_2$ such that $2\eta_2=\tau_2$ and consider the map
\[ \Phi\colon E\times E \to E\times E, \quad (x,y)\mapsto (x+y+\eta_2, x-y+\eta_2)\]
which is a degree $4$ isogeny of $E\times E$ composed with the translation by $(\eta_2, \eta_2)$.

A direct computation shows that $\Phi$ induces the map $\Psi$ in the following diagram
\[
 \begin{tikzcd}
  E\times E \dar{4:1}[swap]{\Phi} \rar{4:1} & E_1\times E_1 \dar{2:1}[swap]{\Psi}\arrow{dr}{\Pi}\\
 E\times E \rar{/e_1}& (E\times E)/\langle e_1\rangle \rar{2:1} & X_{\min}
 \end{tikzcd}.
\]
Now consider the automorphism $\sigma$ of $E\times E$ given by $\sigma( x,  y ) = ( y+\tau_2, x)$. It descends to an automorphism $\bar\sigma$ of $E_1\times E_1$ of order $4$ which satisfies   $\Psi\circ \bar \sigma= \bar e_2\circ \Psi$, where $\bar e_2$ is the induced action on $(E\times E)/\langle e_1\rangle$. Hence $X_{\min}=(E_1\times E_1)/<\bar\sigma>$ and $\Pi$ is an \'etale  $\IZ/4$-cover.

Recall that by construction the pullback $\Gamma$ of $D_{\min}$ to $E\times E$  is the union of $\Delta=\{x-y=0\}$ and of $e_2\Delta=\{x+y+\tau_2=0\}$. Hence  $\Phi^*\Gamma$ is the union for $\tau\in E[2]$ of the four curves $E\times \{\tau\}$  and of the four curves $\{\tau \}\times E$. Thus 
\[\Pi^*D_{\min} =E_1\times \{0\}+E_1\times \{\bar \tau_2\}+\{0\}\times E_1+\{\bar \tau_2\}\times E_1,\]
where $\bar \tau_2$ is the image $\tau_2$ in $E_1$. 

Now let $\tilde Y\to Y_{\min}:=E_1\times E_1$ be the blow up in the four nodes of $\Pi^*D_{\min}$ and $\tilde Y\to Y$ the contraction of the four elliptic curves with self-intersection $-2$. This results in a diagram
\begin{equation}\label{diag: cover and resolution}
\begin{tikzcd}[row sep = small]
  {} & \tilde Y\arrow{dr}{q}\arrow{dl} \dar\\
  Y_{\min}\dar & \tilde X\arrow{dr}\arrow{dl}& Y\dar \\
  X_{\min}& & X
 \end{tikzcd} 
\end{equation}

where all vertical arrows are \'etale  $\IZ/4$-covers.
By Lemma \ref{lem: pi1 contraction} applied to $q$ the surface  $Y$ is simply connected, since $\pi_1(\tilde Y ) = \pi_1(Y_{\min}) = \pi_1(E_1\times\{0\})\times \pi_1(\{0\}\times E_1)$, and thus $Y$ is the universal cover of $X$ and $\pi_1(X) =\IZ/4$.

\refenum{ii}
As in the proof of \refenum{i} we use  the notation of Example  \ref{ex:godeaux-bi-elliptic}. We set $Y_{\min}=(E\times E)/\langle e_1\rangle$ and  let $\Pi\colon Y_{\min}\to X_{\min}$ be the quotient map for the automorphism $\bar e_2$ of $Y_{\min}$  induced  by $e_2$. By construction 
$\Pi^*D_{\min}$ is the orbit of the image of the diagonal under the action of $\bar e_2$ and as such consists of the three components $\bar\Delta$, $\bar e_2 \bar\Delta$, and $2\bar e_2 \bar\Delta$, which intersect pairwise in one node.

Blowing up the nodes of $\Pi^*D_{\min}$ and contracting the elliptic curves we get a diagram as in \eqref{diag: cover and resolution} where now the vertical maps are \'etale  $\IZ/3$-covers. To conclude, it remains to show that $Y$ is simply connected, or equivalently, by Lemma \ref{lem: pi1 contraction}, that $\pi_1(Y_{\min})$ is generated by $\pi_1(\bar\Delta)$, $\pi_1(\bar e_2\bar\Delta)$  and $\pi_1(2\bar e_2\bar\Delta)$. Since all groups involved are abelian, we work in integral homology.

Identifying $ H_1(E\times E) = \IZ[\zeta]^2\subset \IC^2$ we find as subgroups of $\IC^2$:
\begin{align*}
 H_1(Y_{\min}) & = \langle (1,0), (\zeta, 0), (0,1), (0,\zeta), (\tau_1, \tau_1)\rangle_\IZ,\\\
  H_1(\bar \Delta) & = \langle (1,1), (\zeta, \zeta), (\tau_1, \tau_1)\rangle_\IZ,\\
  \intertext{and because $\bar e_2$ acts in homology by multiplication with $\zeta$ in the second variable, $\zeta\tau_1 = \tau_1+\zeta$ and $\zeta^2+\zeta+1=0$,}
  H_1(\bar e_2\bar \Delta) & = \langle (1,\zeta), (\zeta, \zeta^2), (\tau_1, \tau_1+\zeta)\rangle_\IZ,\\
  H_1(2\bar e_2\bar \Delta) & = \langle (1,\zeta^2), (\zeta, 1),
  (\tau_1,\tau_1 -1)\rangle_\IZ.
\end{align*}
Thus the subgroup generated by the three sub-lattices contains the element $(0,1) = (\tau_1, \tau_1)-(\tau_1, \tau_1-1)$ and then we easily find all generators of $H_1(Y_{\min})$, so that $\pi_1(Y)$ is trivial.
\end{proof}

Before we address Godeaux surfaces of type \caseR\ we state  a general result that was proved by Reid in the classical case.
\begin{lem}\label{lem: no-Z22} Let $X$ be a Gorenstein stable Godeaux surface. 
Then $X$ has no \'etale $(\IZ/2)^2$-cover. 
\end{lem}
\begin{proof}
 We argue by contradiction, following the  argument  used in \cite{reid78}.    Assume $p\colon Y\to X$ is an \'etale $(\IZ/2)^2$-cover and let $\{0, \eta_1,\eta_2,\eta_3\}$ be the kernel of $p^*\colon \Pic(X)\to\Pic(Y)$.  Since $h^2( K_X+\eta_i)=h^0(\eta_i)=0$, $h^0(K_X+\eta_i)\ge \chi(K_X)=1$ and we may find non-zero sections $\sigma_i\in H^0(K_X+\eta_i)$, $i=1,2,3$.
  The curves $D_i$ defined by the $\sigma_i$ are irreducible and meet each other  transversally at one point, since $D_iD_j=K_XD_i=1$ and $K_X$ is ample. We have $\sigma_i^2\in H^0(2K_X)$, for $i=1,2,3$. 
Since $P_2(X)=2$ by Riemann--Roch  \cite[Thm. 3.1]{liu-rollenske16}
we have a relation $\sum \lambda_i \sigma_i^2=0$. By this relation, the $D_i$ all have a common point $P$, which is smooth for each of them.
It follows that $\eta_3|_{D_3}=(\eta_1-\eta_2)|_{D_3}=(D_1-D_2)|_{D_3}=0$. Consider a desingularisation $\epsilon\colon \wt X \to X$ and let $Z \to \wt X$ be  the (connected!) double cover   given by  $\epsi^*\eta_3$: then the preimage in $Z$ of $\epsi^*D_3$ has two  connected components with self-intersection equal to 1, contradicting the Index Theorem.
\end{proof}

\begin{lem} \label{lem:pi1-R}
Let $X$ be a normal Godeaux surface of type \caseR. Then $\pi_1(X)$ is a finite abelian group of order $d\ge 3$.
\end{lem}
\begin{proof}
As usual, let $\epsilon \colon \wt X\to X$ be the minimal resolution. Denote by $p\colon \tilde X\to B$ the Albanese map, let $F$ be a general fibre of $p$  and let $d=F\wt D$, so that   induced map  $a\colon\wt D\to B$ is a degree $d$ isogeny.
Then by Lemma \ref{lem: pi1 contraction} we have 
\[\pi_1(X) \isom \frac{\pi_1(\tilde X)}{\pi_1(\tilde D)}\isom \frac {\pi_1(B)}{a_*\pi_1(\tilde D)},\]
which is a finite abelian group of order $d$. 

Since $K_{\tilde X}F = -2$ and $K_{\tilde X} + \tilde D$ is ample on $F$,  we have
 $0<F(K_{\wt X}+\wt D)=d-2$, and we  get $d\ge 3$.
\end{proof}

\begin{prop}\label{prop:torsion-bound-normal}
Let $X$ be a normal Godeaux surface of type \caseR. Then $\pi_1(X)$  is cyclic of order  $d\le 5$.
\end{prop}
\begin{rem}
An example of type \caseR\  with $T(X)$ of order $5$ appears  in   \cite[Ex.~2.14]{Lee00a}. It is constructed by specializing the general construction of  Godeaux surfaces with torsion $\IZ/5$, hence it is smoothable.  
An example with $T(X)$ of order $4$ appears in \cite[Example $(1)$ in \S 6.3]{MLP15} . Again, this is constructed as a degeneration of smooth Godeaux surfaces with $\IZ/4$ torsion.
\end{rem}

\begin{proof} 
 By Lemma \ref{lem:pi1-R} the group $\pi_1(X)$  is finite abelian  of order $d$ and by Lemma \ref{lem: no-Z22} it is sufficient to prove that $d\leq 5$. Consider the connected \'etale cover $Y\to X$  with Galois group $G=\pi_1(X)$; let $f\colon \wt Y\to \wt X$  be the induced cover and set $\Gamma:=f^*\wt D$ and  $ M:=f^*(K_{\wt X}+\wt D)=K_{\wt Y}+\Gamma$. In fact, we have seen in the proof of Lemma \ref{lem:pi1-R} that $\tilde Y = Y\times_B \tilde D$ where $\tilde D\to B$ is induced by the Albanese map. Thus $\Gamma=\Gamma_1+\dots+\Gamma_d$, where the $\Gamma_i$ are smooth disjoint sections of the Albanese map with $\Gamma_i^2=-2$. 
  Recall also that $M\wt F=d-2>0$.

For the reader's convenience we break  the proof  into  steps:
\begin{description}
\item[Step 1] \textit{$h^0(M)=d-1$ and  the curves $\Gamma_i$ are  not contained in the fixed part of $|M|$.}\par
For $i=1,\dots, d$, consider the adjunction sequence:
$$0\to K_{\wt Y}\to K_{\wt Y}+\Gamma_i\to K_{\Gamma_i}=\OO_{\Gamma_i}\to 0.$$
The coboundary map $H^0(K_{\Gamma_i})\to H^1(K_{\wt Y})$ is dual to $H^1(\OO_{\wt Y})\to H^1(\OO_{\Gamma_i})$, hence it is an isomorphism, since $\Gamma_i$ is a section of the Albanese map.
The claim now follows from the long exact sequence in cohomology of 
\[0\to K_{\wt Y}\to M=K_{\wt Y}+\Gamma \to \oplus_{i=0}^d K_{\Gamma_i}\to 0.\]

\item[Step 2] \textit{$|M|$ has no fixed part.}\par
Assume by contradiction  that $|M|=|P|+Z$, with $Z>0$. The group  $G$ acts on $|M|$ by construction, hence $Z$ is a $G$-invariant divisor. In addition, since    $G$ is abelian,  $H^0(M)$ splits as a direct sum of eigenspaces under  the $G$-action, and so there exist a $G$-invariant divisor $M_0=P_0+Z\in |M|$. It follows that $P_0$ is also invariant;   the image $C_0$   of $M_0$ in $X$ is a curve in $|K_X+\eta|$ for some $\eta\in T(X)$.  Since the only curves contracted by the map $\wt Y\to \wt X\to X$ are the $\Gamma_i$,  no component of $Z$ is  contracted by this map because of Step 1 and not every component of $P_0$ is contracted because $P_0$ moves  linearly. So, $C_0$ is reducible with  $K_XC_0=1$, a contradiction since $K_X$ is ample.

\item[Step 3] \textit{If $d\ge 4$, then $|M|$ has no base points.}\par 
The base locus of $|M|$ is finite by Step 2. Assume by  contradiction  that $|M|$ has a base  point $P$. Then, because $G$ acts on $|M|$, all the $d$ points in the  orbit of $P$  are base points of $|M|$. Since $M^2=d$ and $|M|$ moves, we conclude that $|M|$ has precisely $d$ simple base points, and so it  is a pencil, namely $d-1=2$ by Step 1.

\item[Step 4] {$d\le 5$}\par
Assume by  contradiction $d\ge 6$. 
Consider the morphism $h\colon \wt Y\to \pp^{d-2}$ given by $|M|$ and denote by $\Sigma$ the image of $h$, which is a surface because $|M|$ has no base points by Step 3 and $M^2>0$. Since a non-degenerate irreducible surface in $\pp^k$ has degree $\ge k-1$,  we have:
$$d=M^2\ge \deg \Sigma \deg h \ge  (d-3)\deg h,$$
 hence either $\deg h=1$, or  $d=6$ and $\deg h=2$.  
 
Assume $\deg h=1$. We have $MK_{\wt Y}=M^2-M\Gamma=d-0=d$, hence by the adjunction formula the general $M\in |M|$ is smooth of genus $d+1$.  The morphism $h$ maps $M$ birationally to a curve of degree $d$ in $\pp^{d-3}$. So Castelnuovo's bound \cite[III, \S2]{ACGH} gives
$g(M)\le 4$ if $d=6$ and  $g(M)\le 3$ if $d\ge 7$, so we have reached a contradiction.

Finally assume $\deg h=2$ and $d=6$. In this case $\Sigma$ is either a rational normal cubic surface or the cone over the twisted cubic (see e.\,g.\ \cite{eisenbud-harris87}), and therefore it has a unique ruling by lines. Denote by $|\Phi|$ the moving part of  the  pullback to $\wt Y$ of this  ruling:  it  is a $G$-invariant pencil and it is not composed with the Albanese map. Indeed, if  $F$  is a fibre of the Albanese map of $\wt Y$, then $FM=d-2=4$, hence $F$ is mapped to a curve of degree $\ge 2$. 

The divisor $M-2\Phi$ is effective, because there is a hyperplane section of $\Sigma$ that contains two of the rulings. Hence $6=M^2\ge 2M\Phi=2K_{\wt Y}\Phi$, because $\Gamma$ is contracted by $h$; thus $M\Phi=K_{\wt Y}\Phi\le 3 $. By the Index Theorem  $\Phi^2\le \frac{(M \Phi)^2}{M^2}<2$. If $\Phi^2=1$, then $|\Phi|$ has one simple base point, that must be fixed by $G$, since $|\Phi|$ is $G$-invariant, contradicting the fact that $G$ acts freely. So  $\Phi^2=0$ and $|\Phi|$ is a free pencil. 
By  parity $ M\Phi=K_{\wt Y}\Phi=2$, and so  $|\Phi|$ is a free pencil of genus $2$ curves stable under a free action of a group of order 6, but this is impossible  \cite[Lem.~2.2]{CMLP07}. 
 \end{description}
\end{proof}

\section{Non-normal stable Godeaux surfaces}
\subsection{Normalisation and glueing: starting point of the classification}\label{ssec: glue}
We will now show how a combination of the results of \cite{FPR15a} and Koll\'ar's glueing principle leads to a classification of non-normal Gorenstein numerical Godeaux surfaces.

Let $X$ be a stable surface and $\pi\colon \bar X\to X$ its normalisation. Recall that the non-normal locus $D\subset X$ and its preimage $\bar D\subset \bar X$ are pure of codimension $1$, i.\,e. curves. Since $X$ has ordinary double points at the generic points of $D$ the map on normalisations $\bar D^\nu\to D^\nu$ is the quotient by an involution $\tau$. 
Koll\'ar's glueing principle says that $X$ can by be uniquely reconstructed from $(\bar X, \bar D, \tau\colon \bar D^\nu\to \bar D^\nu)$ via the following two push-out squares:
\begin{equation}\label{diagr: pushout}
\begin{tikzcd}
    \bar X \dar{\pi}\rar[hookleftarrow]{\bar\iota} & \bar D\dar{\pi} & \bar D^\nu \lar[swap]{\bar\nu}\dar{/\tau}
    \\
X\rar[hookleftarrow]{\iota} &D &D^\nu\lar[swap]{\nu}
    \end{tikzcd}
\end{equation}

Applying this principle to non-normal Gorenstein stable Godeaux surfaces, we deduce by \cite[Thm.~5.13]{KollarSMMP} and \cite[Addendum in Sect.3.1.2]{FPR15a} that a triple $(\bar X, \bar D, \tau)$ corresponds to a Gorenstein stable Godeaux surface if and only if the following four conditions are satisfied:
\begin{description}
\item[lc pair condition] $(\bar X, \bar D)$ is an lc pair, such that $K_{\bar X}+\bar D$ is an ample Cartier divisor.
\item[$K_X^2$-condition] $(K_{\bar X}+\bar D)^2=1$.
 \item[Gorenstein-glueing condition] $\tau\colon \bar D^\nu\to \bar D^\nu$ is an involution that restricts to a fixed-point free involution on the preimages of nodes.
 \item[Godeaux condition] The holomorphic Euler-characteristic of the non-normal locus $D$ is  $\chi(D) = 1-\chi(\bar X)+\chi(\bar D)$.
\end{description}
In \cite{FPR15a} we classified  lc pairs $(\bar X, \bar D)$ (with $\bar D\neq0$) satisfying the first and second condition, and also showed (Thm.~3.6 loc.\ cit.) that an involution $\tau$ such that $(\bar X, \bar D, \tau)$  satisfies all four conditions can only exist in three cases, labelled \casedP, \caseP{}, and \caseE{+}.

We will now classify the possible involutions $\tau$ for \caseP{} and \caseE{+}, and investigate the geometry and topology of the resulting Gorenstein stable Godeaux surfaces; the case \casedP\ has been treated in \cite{rollenske16} and we recall the results in Section \ref{sect: case dP}.

\subsection{Case \caseP{}.}
\label{sect: case P}
In this case $\bar X = \IP^2$ and $\bar D$ can be any nodal plane quartic, so $p_a(\bar D)=3$ and the Godeaux condition is satisfied if and only if $p_a(D)=3$. 

We  fix some notation first  and then we write  down three constructions, each depending on one parameter.

Fix $Q_1,\dots, Q_4\in\pp^2$ points in general position and denote by $\kc$ the pencil of conics through these points.  Denote by $L_{ij}$ the line through $Q_i$ and $Q_j$;  the conics $L_{12}+L_{34}$, $L_{13}+L_{24}$ and $L_{14}+L_{32}$ are the only reducible conics of $\kc$.  Any permutation $\gamma\in S_4$ of $Q_1, \dots Q_4$ determines an  automorphism of $\pp^2$ that in turn induces an automorphism of $\kc$;   we still denote all these automorphism by the same letter.  The action of $S_4$ on $\kc$  is not faithful: indeed a permutation induces the identity on $\kc$ if and only if it fixes the three reducible conics. So the kernel of the map $S_4\to \Aut(\kc)$ is the subgroup $H=\{\id, (12)(34), (13)(24), (14)(32)\}$ and therefore the image of $S_4$ is  a subgroup of $\Aut(\kc)$ isomorphic to $S_3$. 

\begin{rem}\label{rem: birapporto}
If $C\in \kc$ is a smooth member, then  the cross-ratio $\beta_C\in \IC-\{0,1\}$ of the points $Q_1,Q_2,Q_3,Q_4$ is well defined. It is a classical fact (cf. for instance \cite[Es.~4.24]{FFP11}) that $\beta_C$ determines $C\in\kc$. 
We denote by $j(C):=\frac{(\beta_C^2-\beta_C+1)^3}{\beta_C^2(\beta_C-1)^2}$ the $j$-invariant of the unordered quadruple $\{Q_1,\dots, Q_4\}$ on  $C\cong \pp^1$.  So we have a rational function $j\colon \kc\to \pp^1$ such that the restriction of $j$ to the set of irreducible elements of $\kc$ is the quotient map  for the $S_4$-action.  

From this one can easily deduce that different general choices of parameter in the constructions below give non-isomorphic surfaces.
\end{rem}

\begin{description}
 \item[Case \caseP{1}] Let $\sigma=(1234)\in S_4$. The permutation $\sigma$ acts on $\kc$ as an involution that fixes $L_{13}+L_{24}$ and a smooth conic $C_0$. Given $C\in \kc$ such that $C\ne \sigma(C)$, we set $\bar D:=C+\sigma(C)$ and we take as $\tau$ the involution of ${\bar D}^{\nu}=C\sqcup \sigma(C)$ that identifies $C$ with $\sigma(C)$ via the restriction of $\sigma$. Clearly $\tau$ satisfies the Gorenstein condition, and the stable surface thus obtained has a degenerate cusp at the image point of $Q_1,\dots, Q_4$. If $C$ is irreducible, then  the double curve  $D$ is a rational curve with a unique (semi-normal) quadruple point, hence $p_a(D)=3$ and the Godeaux condition is satisfied. We give a graphical representation in Figure \ref{fig: P.1 construction}.

If $C=L_{12}+L_{34}$ 
is reducible, the double locus splits in two components, the quadruple point persists and  an additional   node appears, so we have $p_a(D)=3$ also in this case.  In the classification of surfaces with normalisation $(\IP^2, \text{4 lines})$ (see \cite{FPR15b}, Table~1) such a configuration occurs only once, in the case $X_{1.4}$.
 
 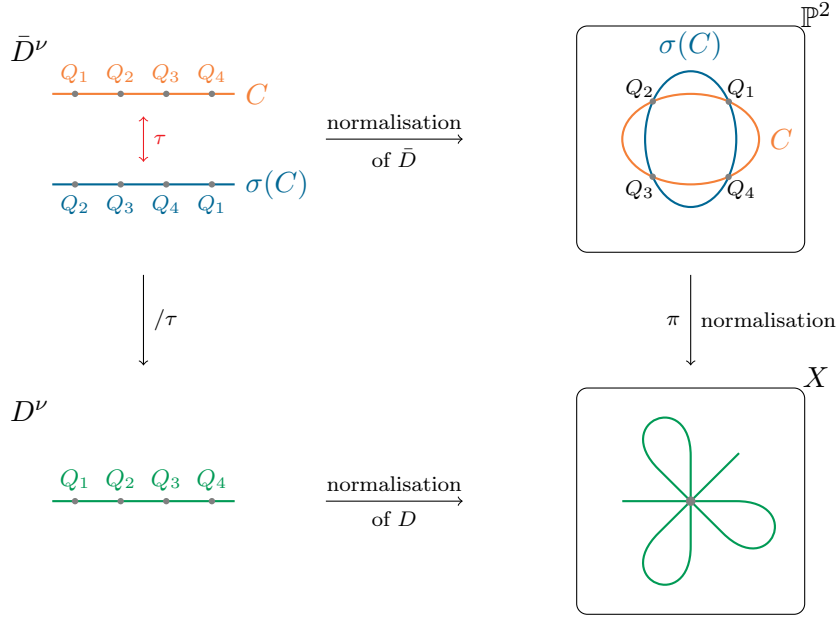
\begin{figure}

 \begin{tikzpicture}
[pfeil/.style = {->, every node/.style = { font = \scriptsize}},
scale = .6]

\begin{scope}[xshift = 12cm, yshift = 8cm]
\node at (2.75, 2.75) {$\IP^2 $};

\draw[rounded corners] (-2.5, 2.5) rectangle (2.5, -2.5);
\draw[C1col, name path = C1] (0,0) ellipse (1 and 1.5 ) ++(0,1.5) node[above] {$\sigma(C)$};
\draw[C2col, name path = C2] (0,0) ellipse (1.5 and 1 ) ++(1.5,0) node[right] {$C$};;

\fill [name intersections={of= C1 and C2, name=Q}]
[Qs]
\foreach \s in {1,...,4}
{(Q-\s) circle (2pt) ++(-45+\s*90:.4) node  { $Q_\s$}};

\draw[pfeil] (0, -3) to  node[left] {$\pi$} node[right] { normalisation} ++(0,-2);
\end{scope}

\begin{scope}[xshift = 0cm, yshift = 8cm]
\node at (-2.5,2) {$\bar D^\nu$};
\draw[ pfeil,<->, Red] (0,.5) to node[right]{$\tau$} (0,-.5);
\draw[C1col] (-2, -1)-- (2,-1) node[right]{$\sigma(C)$};
\draw[C2col] (-2, 1)-- (2,1) node[right]{$C$};

\fill [Qs]
\foreach \s in {1,...,4}
{(-2.5+\s,1) circle (2pt) node[C2col, above]  { $Q_\s$}};

\fill [Qs]
\foreach \s in {2,...,4}
{(-2.5+\s-1,-1) circle (2pt) node[C1col, below]  { $Q_\s$}}
(1.5,-1) circle (2pt) node[C1col, below]  { $Q_1$};

\draw[pfeil] (4,0) to  node[below] { of $\bar D$} node[above] { normalisation} ++(3,0);
\draw[pfeil] (0, -3) to  node[left] {} node[right] { $/\tau$} ++(0,-2);
\end{scope}

\begin{scope} 
\node at (-2.5,2) {$ D^\nu$};
\draw[C12col] (-2, 0)-- (2,0) node[right]{
};
\fill [Qs]
\foreach \s in {1,...,4}
{(-2.5+\s,0) circle (2pt) node[C12col, above]  {\footnotesize $Q_\s$}};
\draw[pfeil] (4,0) to  node[below] { of $ D$} node[above] { normalisation} ++(3,0);
\end{scope}

\begin{scope}
[xshift = 12cm, yshift =0cm,  looseness=5]
\node at (2.75, 2.75) {$X$};
\draw[rounded corners] (-2.5, 2.5) rectangle (2.5, -2.5);
 \path (0,0) coordinate(P) ;
 
\draw[C12col]
\foreach \x in {0,270, 135} {(P)  -- ++(\x:1) to[out=\x, in =\x-45]  (\x-45:1) -- (P)}
\foreach \x in {180, 45} {(P)  -- ++(\x:1.5)};
 \fill [Qs] (P) circle (3pt);
\end{scope}
\end{tikzpicture}
\caption{The general surface in \caseP{1}}\label{fig: P.1 construction}
\end{figure}

\item[Case \caseP{2}]  Set  $\rho=\sigma^2=(13)(24)\in S_4$. The action of $\rho$ on $\kc$ is trivial, namely $\rho$ preserves the conics of $\kc$.
We take $\bar D=C+L_{13}+L_{24}$, with  $C\in \kc$ distinct from $L_{13}+L_{24}$. The involution $\tau$ on ${\bar D}^{\nu}=C\sqcup L_{13}\sqcup L_{24}$ is defined on $C$ as the restriction of $\rho$;  in addition, $\tau$ switches   $L_{13}$ with $L_{24}$, identifying them via the only isomorphism such that:
\[R\mapsto Q_2, \quad Q_1\mapsto Q_4, \quad Q_3\mapsto R',\]
where $R$  denotes the point $L_{13}\cap L_{24}$   if we consider it on $L_{13}$,  and   $R'$ denotes the same point if we consider it on $L_{24}$; the involution $\tau$ satisfies the Gorenstein condition.

If $C$ is irreducible, then the double locus $D$ of the corresponding stable surface  has two components; the image of $C$, which is a nodal rational curve, and the image of $L_{13}$ and $L_{24}$, which is a rational curve with a unique (semi-normal) triple point; the components meet at the singular points,  which combine to a semi-normal quintuple point. In particular, $p_a(D)=3$ and $X$ is  a Godeaux surface.

When $C$ is reducible, then  the double curve remains the same but the surface develops an additional degenerate cusp on the nodal components, which locally is isomorphic to a cone over a nodal plane cubic, and we have a Godeaux surface also in this case.
More precisely, if $C=L_{14}+L_{23}$, then the Gorenstein condition is still verified and one can check that the resulting Godeaux surface is case $X_{1.1}$ in Table~1 of \cite{FPR15b}, while if $C=L_{12}+L_{34}$ then we get case $X_{1.2}$.

\item[Case \caseP{3}] We take $\bar D$ as in case \caseP{2}, but we choose a different involution: we let $\tau$ be the involution that acts  on $L_{13}\sqcup L_{24}$ as in case \caseP{2} and on  $C$ via the permutation $(12)(34)$. The difference to the previous family is that here the involution on $C$ does not preserve the intersection $L_{13}\cap C$. 
When $C$ is smooth the non-normal locus looks like  in case \caseP{2}, but the way it is glued into the surfaces, encoded by the involution,  is different.

 If $C=L_{14}+L_{23}$, then the Gorenstein condition is satisfied and we get  case $X_{1.3}$   in Table~1 of \cite{FPR15b}. 
 
 If $C=L_{12}+L_{34}$, then the Gorenstein condition is not satisfied, because $\tau$ fixes the preimages in $L_{12}\sqcup L_{34}$ of the singular point of $C$. The non-normal locus has $3$ components
 and at the newly developed node $X$ has a non-Gorenstein singularity of index $2$. 
\end{description}

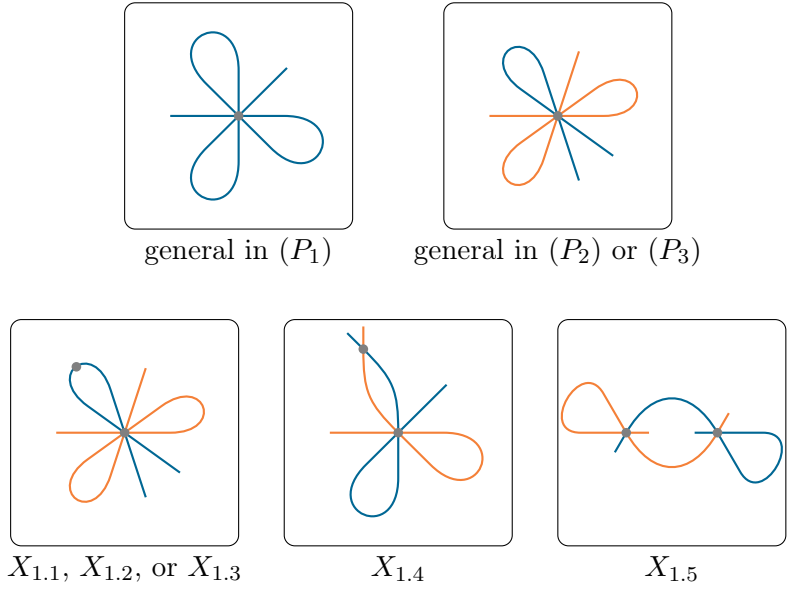
\begin{figure}
 \begin{tikzpicture}
[C2/.style ={thick, Orange},
C1/.style={thick, MidnightBlue},
Qs/.style ={black!50!white, opacity=1, every node/.style={black, opacity=1, font  = \scriptsize}},
scale = .6]

\begin{scope}
[xshift = 3.5cm, yshift =7cm,  looseness=5]
\node at (0,-3) {general in \caseP{2} or \caseP{3}};
\draw[rounded corners] (-2.5, 2.5) rectangle (2.5, -2.5);
 \path (0,0) coordinate(P) ;
 
\draw[C2]
\foreach \x in {0, 6*36} {(P)  -- ++(\x:1) to[out=\x, in =\x+36]  (\x+36:1) -- (P)}
\foreach \x in {180, 72} {(P)  -- ++(\x:1.5)};
\draw[C1]
\foreach \x in {3*36} {(P)  -- ++(\x:1) to[out=\x, in =\x+36]  (\x+36:1) -- (P)}
\foreach \x in {8*36,9*36} {(P)  -- ++(\x:1.5)};
 \fill [Qs] (P) circle (3pt);
\end{scope}

\begin{scope}
[xshift =-3.5cm, yshift = 7cm,  looseness=5]
\node at (0,-3) {general in \caseP{1}};
\draw[rounded corners] (-2.5, 2.5) rectangle (2.5, -2.5);
 \path (0,0) coordinate(P) ;
 
\draw[C1]
\foreach \x in {0,270, 135} {(P)  -- ++(\x:1) to[out=\x, in =\x-45]  (\x-45:1) -- (P)}
\foreach \x in {180, 45} {(P)  -- ++(\x:1.5)};
 \fill [Qs] (P) circle (3pt);
\end{scope}

\begin{scope}
[xshift =-6cm,looseness=5]
\node at (0,-3) {$X_{1.1}$, $X_{1.2}$, or $X_{1.3}$};
\draw[rounded corners] (-2.5, 2.5) rectangle (2.5, -2.5);
 \path (0,0) coordinate(P) ;
 
\draw[C2]
\foreach \x in {0, 6*36} {(P)  -- ++(\x:1) to[out=\x, in =\x+36]  (\x+36:1) -- (P)}
\foreach \x in {180, 72} {(P)  -- ++(\x:1.5)};
\draw[C1]
\foreach \x in {3*36} {(P)  -- ++(\x:1) to[out=\x, in =\x+36]  (\x+36:1) -- (P)}
\foreach \x in {8*36,9*36} {(P)  -- ++(\x:1.5)};
 \fill [Qs] (P) circle (3pt);
\fill [Qs] (126:1.8) circle (3pt);
\end{scope}

\begin{scope}
[ looseness=5]
\node at (0,-3) {$X_{1.4}$};
\draw[rounded corners] (-2.5, 2.5) rectangle (2.5, -2.5);
 \path (0,0) coordinate(P) ;
 
\draw[C1]
\foreach \x in { 270} {(P)  -- ++(\x:1) to[out=\x, in =\x-45]  (\x-45:1) -- (P)}
\foreach \x in {45} {(P)  -- ++(\x:1.5)}
 (P) to [out=90, in = -45, looseness = 1.2] (112.5:2) -- ++(135:0.5);

\draw[C2]
\foreach \x in {0} {(P)  -- ++(\x:1) to[out=\x, in =\x-45]  (\x-45:1) -- (P)}
\foreach \x in {180} {(P)  -- ++(\x:1.5)}
(P) to [out=135, in = -90, looseness = 1.2] (112.5:2) -- ++(90:0.5);

 \fill [Qs] (P) circle (3pt) 
 (112.5:2) circle (3pt) ;
\end{scope}

\begin{scope}
[xshift =6cm]
\node at (0,-3) {$X_{1.5}$};
\draw[rounded corners] (-2.5, 2.5) rectangle (2.5, -2.5);
 \path (-1,0) coordinate(P) ;
 \path (1,0) coordinate(Q) ;
 
\draw[C1]
(P) ++(240:.5) -- (P) to[out = 60, in = 120, looseness=1.5] (Q) -- ++(300:1) to[out=300, in = 0, looseness = 2.5] (2,0) -- (Q) -- ++(180:.5); 
\draw[C2]
(Q) ++(60:.5) -- (Q) to[out = 240, in = 300, looseness=1.5] (P) -- ++(120:1) to[out=120, in = 180, looseness = 2.5] (-2,0) -- (P) -- ++(0:.5);

\fill [Qs] (P) circle (3pt);
\fill [Qs] (Q) circle (3pt);
\end{scope}
\end{tikzpicture}
\caption{Non-normal loci in case \caseP{} }\label{fig: P double loci}
\end{figure}

\begin{prop}\label{prop: Godeaux from P}
 Let $X$ be a Gorenstein stable surface with normalisation $\bar X$ isomorphic to $\pp^2$.
 Then either $X$ is obtained as in one of the above cases   \caseP{1}, \caseP{2} or \caseP{3},  or it is  isomorphic to the surface $X_{1.5}$ in Table~1 of \cite{FPR15b}.\end{prop}

\begin{proof}
Let $\mu_1$ be the number of degenerate cusps in $X$,  let $\rho$ be  the number of ramification points of the map $\bar D^\nu \to D^\nu$ and let $\bar \mu$ be the number of nodes of $\bar D$.
Then by  \cite[Lem.~3.5]{FPR15a} we have the equality $\chi(D) = \frac 12\left( \chi({\bar D})-\bar \mu\right)+\frac\rho 4+\mu_1$. Using $\chi(D) = \chi(\bar D) = -2$ we obtain
\begin{equation}\label{eq: nodes on barD} \bar\mu = \frac\rho 2+2\mu_1+2\end{equation}
Thus $\bar D$ has at least one node, which implies $\mu_1\geq 1$ by the classification of Gorenstein slc singularities (see the proof of Lemma 3.5 in \cite{FPR15a}). Also, a plane quartic can have at most $6$ nodes, so in total we get $4\leq \bar\mu\leq6$.

Note that a plane quartic with at least $4$ nodes is reducible, so $\bar D$ consists of four lines, two lines and a conic, two conics, or a line and a nodal cubic. We proceed case by case.

\begin{description}
\item[$\bar D = \text{four general lines}$] This case has been classified in \cite[Sect.~4.2]{FPR15b}. The four surfaces  $X_{1.1}$, $X_{1.2}$, $X_{1.3}$  and $X_{1.4}$ appear as special cases of construction \caseP{1}, \caseP{2} or \caseP{3} and $X_{1.5}$ is listed separately.

\item[$\bar D = \text{a conic and two lines}$] The gluing involution $\tau$ has to preserve the conic and exchange the two lines. By the Gorenstein condition, the involution on the conic cannot fix any of the four intersection points with the pair of lines.

By \eqref{eq: nodes on barD} we have a unique degenerate cusp in $X$, thus $\tau$ cannot interchange the preimages  in $\bar D^\nu$ of the intersection of the two lines.

Now there are two possibilities: either the involution on the conic preserves the intersection with each individual line  and we have case \caseP{2},  or it does not and  we have case \caseP{3}.
\item[$\bar D = \text{ two irreducible conics}$]
By \eqref{eq: nodes on barD} we have a unique degenerate cusp and $\tau$ has no fixed points on $\bar D^\nu$. Thus $\tau$ exchanges the two conics $C$ and $C'$, that is, we have an abstract isomorphism  $\phi = \tau|_C\colon C\to C'$ preserving the four intersection points $Q_1, \dots, Q_4$. We can number the intersection points such that $\phi(Q_i) = Q_{i+1}$ (where $Q_5=Q_1$) because otherwise there would be more than one degenerate cusp.

Now consider the unique automorphism $\sigma$ of $\IP^2$ that acts on the $Q_i$ in the same way as $\phi$ and let $\sigma C$ be the image of $C$ under $\sigma$. 
The composition $\sigma\circ \inverse\phi\colon C'\to  \sigma C$ is an abstract isomorphism of two plane conics fixing four points in the plane. By \cite[Es.~4.24]{FFP11} (see also Remark \ref{rem: birapporto}) it is actually induced by the identity on $\IP^2$, thus $C' = \sigma C$ and $\phi=\sigma$ and we are in case \caseP{1}. 
\item[$\bar D = \text{a nodal cubic and a line}$]  The involution has to preserve the line, because the number of marked points on the two components of $\bar D^\nu$ is different. In particular any $\tau$ fixes the preimage of a node of $\bar D$ and the Gorenstein condition cannot be satisfied.
\end{description}
We have enumerated all possible cases for $\bar D$ and thus concluded the classification.
\end{proof}

\subsection{Case \caseE+}\label{sect: case E}

\subsubsection{Stable Godeaux surfaces and bi-tri-elliptic curves}

 Let $X$ be a stable Godeaux surface such that its normalisation $\bar X$  is the symmetric product of an elliptic curve $E$. We continue to use the notation from \eqref{diagr: pushout}.
 Recall that the conductor locus $\bar D\subset \bar X$ is a stable curve of arithmetic genus two, which is a trisection of the Albanese map $a\colon \bar X =S^2 E\to E$.
 The Godeaux condition
 (cf. \S \ref{ssec: glue})   implies that $\chi(D) = \chi(X) -\chi(\bar X)+\chi(\bar D)=0$, so we get a configuration
 \begin{equation}\label{diag: bitrielliptic}
  \begin{tikzcd}
   {} &\bar D \arrow{dl}{2:1}[swap]{\pi}\arrow{dr}{q}[swap]{3:1}\\
   D && E
  \end{tikzcd},
 \end{equation}
 where $D$ and $E$ are smooth elliptic curves. 
 This motivates the following definition (cf. \cite{FPR16a}). 
 \begin{defin}
A bi-tri-elliptic configuration is a diagram of curves as in \eqref{diag: bitrielliptic} 
where $D$, $E$ are elliptic curves, $\bar D$ is a stable curve of genus $2$ and $q$ and $\pi$ are finite with $\deg q =3$ and $\deg \pi=2$. The map $\pi$ is the quotient by an involution which we call the bi-elliptic involution of the configuration.

A  curve of genus $2$ is called bi-tri-elliptic curve if it admits a bi-tri-elliptic configuration.
\end{defin}

It turns out that this configuration of curves determines the surface $X$ uniquely.
\begin{prop}\label{prop: Godeaux from E+ vs. bitrielliptic configurations}
There is a bijection (up to isomorphism) between bi-tri-elliptic configurations and Gorenstein stable Godeaux surfaces with normalisation the symmetric product of an elliptic curve.
\end{prop}
\begin{proof}
 We have seen above that every Gorenstein stable Godeaux surfaces with normalisation the  symmetric product of an elliptic curve gives rise to a uniquely determined bi-tri-elliptic configuration. We now need to show that we can recover the surface $X$ just from the datum \eqref{diag: bitrielliptic}.
 The main step is
 \begin{lem}\label{lem:triple}
 Let $q\colon \bar D\to E$ be a triple cover, with $E$ an elliptic curve and $\bar D$ a nodal curve of genus 2. Then $\bar D$ embeds in $S^2E$ as a trisection of the Albanese map $S^2E\to E$.
\end{lem}
\begin{proof} The trace map gives a splitting $q_*\OO_{\bar D}=\OO_E\oplus \ke^{\vee}$, where $\ke$ is a rank 2 vector bundle on $E$. 
By   \cite[Thm~1.3]{CE96} (cf. also, \cite{miranda85}),  the curve   ${\bar D}$ is embedded in $\pp(\ke)$ as a trisection and $\chi(\ke^\vee)  = \chi(\ko_{\bar D}) = -1$, so by Riemann--Roch $\deg \ke = 1$.
The projection formula  $$q_*\omega_{\bar D}=\omega_E\oplus \omega_E\otimes \ke=\OO_E\oplus \ke$$ gives $h^0(\ke)=1$; analogously, given a point $P\in E$, we get $0=h^0(\omega_{\bar D}(-q^*P))=h^0(\OO_E(-P))+h^0(\ke(-P))$, hence $h^0(\ke(-P))=0$. It follows that the only section of $\ke$ vanishes nowhere and it gives  an exact sequence $0\to \OO_E\to \ke \to \kl\to 0$, where $\kl$ is a line bundle of degree 1. The sequence is not split, since $h^0(\ke)=1$. This proves that $\pp(\ke)\cong S^2E$, see e.\,g.\ \cite[\S1]{catanese-ciliberto93}.
\end{proof}

Write $\bar X= S^2E$ and  let $\Phi$ be a  fibre of the Albanese map $a\colon \bar X\to E$; since $h^0(\ke)=1$, there    is a section  $C_0$ of $a$ with $C_0^2=1$. The classes of $C_0$ and $\Phi$ generate $\Pic(\bar X)$ up to numerical equivalence and $K_{\bar X}$ is linearly  equivalent to $-2C_0+\Phi$.   Since $\bar D\Phi=3$ and  $2=(K_{\bar X}+\bar D)\bar D$ by adjunction, it follows that $\bar D$ is algebraically  equivalent to $3C_0-\Phi$ and $K_{\bar X}+\bar D$ is algebraically equivalent to $C_0$.

 Denote by $\tau$ the bi-elliptic involution of the bi-tri-elliptic configuration. Then, by Koll\'ar's glueing principle, the triple $(\bar  X, \bar D, \tau)$ gives rise to a stable surface $X$. The Gorenstein and the Godeaux condition (cf. \S \ref{ssec: glue})  are easily checked and we get a Gorenstein stable Godeaux surface.
 
 Clearly, the two constructions are inverse to each other, up to isomorphism.
\end{proof}
Thus to understand Gorenstein stable Godeaux with normalisation a symmetric product of an elliptic curve we need to classify bi-tri-elliptic configurations.

\subsubsection{Classification of bi-tri-elliptic curves}
We now recall  the classification of bi-tri-elliptic configurations from \cite{FPR16a}, where we study more generally $(p,d)$-elliptic configurations on curves of genus $2$.

If $\bar D$ is a stable bi-tri-elliptic curve, then it is either smooth or the union of two elliptic curves meeting at a point. So  the Jacobian $A:=J(\bar D)=\Pic^0(\bar D)$ is compact, that is, $\bar D$ is of compact type, and   $\bar D$ can be embedded in $A$ as a principal polarisation.

Following work of Frey and Kani (see \cite{frey-kani}),  the Jacobian of a stable bi-elliptic curve $\bar D$ of genus 2 of compact type can be described as follows. 
Take  elliptic curves $D$, $D'$ and a subgroup $G\subset D[2]\times D'[2]$ such that $G$ is isomorphic to $(\IZ/2)^2$ and  $G\cap (D\times\{0\})=G\cap(\{0\}\times  D')=\{0\}$.
Then  the double of  the  product polarisation on $D\times  D'$ induces a principal polarisation $\Theta$ on $A=(D\times D')/G$.
If $\bar D\subset A$ is a curve in the class of $\Theta$, then $\bar D$ is stable of genus $2$ of compact type and the map $A\to D/D[2]=D$ restricts to a degree two map $\pi \colon \bar D\to D$.

If moreover $A$ contains  a $1$-dimensional subgroup such that $\bar D\bar F=3$, then the quotient map $A\to A/\bar F:=E$ induces a tri-elliptic map $q\colon \bar D\to E$. To construct $\bar F$ we construct an appropriate $1$-dimensional subgroup $F\subset D\times D'$ resulting in a diagram
\[
 \begin{tikzcd}
  F\rar[ hookrightarrow]{(\phi, \phi')}\dar{h_F}& D\times D'\dar{/G}\\
  \bar F\rar[hookrightarrow]\arrow{dr}[swap]{m:1} &  J(\bar D)\rar{q}  \dar{\pi} & E\\
   & D& 
 \end{tikzcd}.
\]
From this situation we extract a numerical invariant of the bi-tri-elliptic configuration,
\begin{equation}\label{eq: m} m = m(\bar D, \pi,q) := \deg (\bar F\to D ) = \frac{4\deg \phi}{\deg h_F},
\end{equation} 
which we call the twisting number of $(\bar D, \pi, q)$. 
We now give two instances of this construction. 
 \begin{exam}\label{ex: odd}
Let $F$ be an elliptic curve and let $\phi\colon F\to D$ and $\phi' \colon F\to D'$ isogenies such that $\ker \phi\cap \ker \phi'=\{0\}$, so that $\phi\times\phi'\colon F\to D\times D'$ is injective. We abuse notation and denote again by $F$ the image of $\phi\times \phi'$. 
Assume that $\deg\phi+\deg \phi'=6$ and $\deg\phi$ is odd. Set $G:=F[2]\subset D\times D'$: for $i=1,2$  we have $G\cap (D\times\{0\})=G\cap(\{0\}\times  D')=\{0\}$  since $DF=\deg \phi$ and $D'F = \deg \phi'$ are odd. Let $\bar F$ be the image of $F$ in $A:=(D\times D')/G$. If $\Theta$ is the principal polarisation of $A$, then we have $4\Theta\bar F=2(D\times \{0\}+\{0\}\times D')F=12$, namely $\Theta\bar F=3$. 

The twisting number in this example is $m = \deg \phi$. 
\end{exam}

\begin{exam}\label{ex: even}
We proceed as in Example \ref{ex: odd}, but here  $\phi$ and $\phi'$ satisfy  $\deg\phi+\deg \phi'=3$.  It is possible to find  a subgroup $G\subset D[2]\times D'[2]$ of order 4 such that $G\cap (D\times\{0\})=G\cap(\{0\}\times  D')=\{0\}$ and $G\cap F$ has order 2. In fact, there are precisely  four choices for such a subgroup. Again, we denote by $\bar F$ the image of $F$ in $(D\times D')/G$ and we compute $4\bar F\Theta=2F(2(D\times \{0\}+\{0\}\times D'))=12$, that is, $\bar F\Theta=3$. 

The twisting number of the resulting bi-tri-elliptic configuration is $m = 2\deg \phi$. 
\end{exam}
These two examples cover all possible cases.
\begin{prop}[\protect{\cite[Cor. 4.1]{FPR16a}}]\label{prop: classification bi-tri-elliptic}
Let $(\bar D, \pi, q)$ 
 be a bi-tri-elliptic configuration on a stable curve of genus $2$. Then the twisting number $m$ defined in \eqref{eq: m} satisfies $1\le m\le 5$ and there are the following possibilities:
\begin{enumerate}
\item  $m$ is  odd and  the configuration arises as in Example \ref{ex: odd} with $\deg\phi=m$;
\item $m$ is even and  the configuration  arises as in Example \ref{ex: even} with $\deg\phi=\frac m 2$.
\end{enumerate}
\end{prop}
\begin{rem}\label{rem: reducible bi-tri-elliptic}
We now consider the case that in \eqref{diag: bitrielliptic} the bi-tri-elliptic curve is reducible. It is shown in \cite[Sect. 4]{FPR16a} that there is an elliptic curve $D$ admitting an endomorphism $\psi\colon D\to D$ of degree $2$ such that $\bar D = D\cup_0 D $ and the bi-tri-elliptic configuration is given by 
\[ \begin{tikzcd}{} &D \cup_0 D\arrow{dl}{2:1}[swap]{\pi=\id\cup\id}\arrow{dr}{q=\id\cup \psi}[swap]{3:1}\\
   D && D
  \end{tikzcd}.
\]
In \cite[Sect. 4]{FPR16a} it is also shown  that there are exactly ten such configurations, two for each possible twisting number $1\leq m \leq 5$. 
\end{rem}

\subsubsection{Classification of surfaces}
To exploit the connection between bi-tri-elliptic configurations and stable Godeaux surfaces we make the following definition.
\begin{defin}
Let $X$ be a Gorenstein stable Godeaux surface with normalisation the symmetric product of an elliptic curve and let $(\bar D, \pi, q)$ be the associated bi-tri-elliptic configuration. 
We then say $X$ is of type \caseE{m}, where $m$ is the twisting number of $(\bar D, \pi, q)$ defined in \eqref{eq: m}. 
\end{defin}
Then Propositions \ref{prop: Godeaux from E+ vs. bitrielliptic configurations} and  \ref{prop: classification bi-tri-elliptic} immediately yield: 
\begin{prop}\label{prop: classification Godeaux from E+}
 Let $X$ be a Gorenstein stable Godeaux surface with normalisation the symmetric product of an elliptic curve. Then $X$ is of type \caseE{m} for some $1\leq m\leq 5$.
\end{prop}

\subsection{Calculating fundamental groups}\label{sec:fundamental group}
\begin{prop}\label{prop: pi1 case P}
 Let $X$ be a Gorenstein stable Godeaux surface with normalisation $\IP^2$. 
 Then
 \begin{enumerate}
  \item If $X$ is as in case \caseP{1}, including $X\isom X_{1.4}$, then $\pi_1(X) \isom \IZ/4$.
  \item If $X$ is as in case \caseP{2}, including $X\isom X_{1.1}$ or $X\isom X_{1.2}$, then $\pi_1(X) \isom \{1\}$.
  \item If $X$ is as in case \caseP{3}, including $X\isom X_{1.3}$, then $\pi_1(X) \isom \IZ/3$.
  \item If $X\isom X_{1.5}$ then $\pi_1(X) \isom \IZ/5$.
 \end{enumerate}
\end{prop}
\begin{proof}
The proof relies on \cite[Cor.~3.2\,\refenum{iii}]{FPR15b}, which tells us that in the notation of Section \ref{ssec: glue}  we have
\[\pi_1(X)\isom \pi_1(\IP^2)\ast_{\pi_1(\bar D)}\pi_1(D)\isom \frac{\pi_1(D)}{R},\]
where $R=\langle\langle \pi_*\pi_1(\bar D)\rangle\rangle$ is the smallest normal subgroup of $\pi_1(D)$ generated by the image of $\pi_1(\bar D)$.
We make use of the explicit descriptions in Section \ref{sect: case P} and \cite[Table 1]{FPR15b}.

Note that topologically every irreducible component of $\bar D$ is a $2$-sphere; to do the computation we consider the following cell-decomposition:
on every conic we arrange the four intersection points on a great circle, depicted  as an ellipse, and consider these as the $0$-skeleton and $1$-skeleton. The sphere is formed by capping with two $2$-cells.
On every line we choose the three intersection points as $0$-cells, two intervals as $1$-cells, and then glue in a $2$-cell to make up the sphere. The order of the intersection points in the $1$-skeleton will be chosen such that $\pi\colon \bar D\to D$ is easily visualised, 
 that is, the involution is given by the obvious identification of oriented $1$-skeletons.

Below we will in each case give  the map between the $1$-skeletons of $\bar D$ and $D$ and denote by \tikz \fill[black] circle (3pt); a chosen base-point.  We will name paths in the $1$-skeleton of $\bar D$ with lower case letters which under $\pi$ will map to paths in the $1$-skeleton of $D$ marked by the matching upper  case letters. The indices will indicate the irreducible component  the path belongs to. 
For each conic, the two $2$-cells give one relation in the fundamental group, which we will  indicate below the figures; the $2$-cells of the lines do not affect the fundamental group.

Note that the shape of the $1$-skeleton of the non-normal locus $D$ does  in some cases differ from the mnemonic representation in Figure \ref{fig: P double loci}.

\begin{description}
 \item[general in \caseP{1}] The map on $1$-skeletons and the extra relations are as follows: 
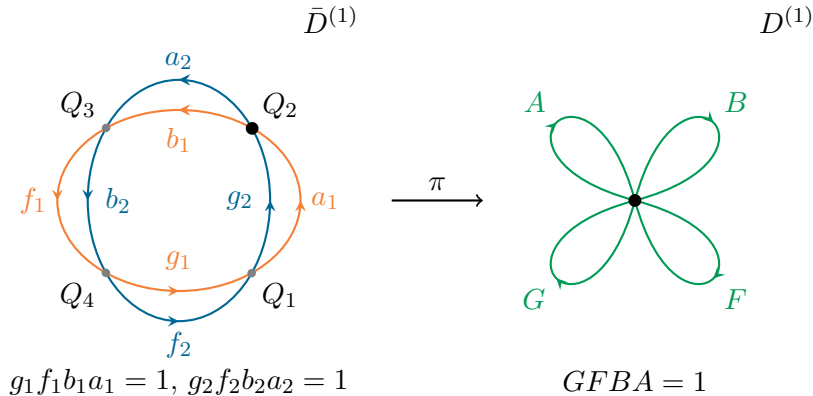
\begin{center}
\begin{tikzpicture}[scale = .8]
\begin{scope}
\draw[C2, name path = C2, 
postaction={decorate,decoration={markings,
mark=at position .25 with {\arrow{stealth}};,
 mark=at position .25 with {\arrow{stealth}};,
mark=at position .75 with {\arrow{stealth}};,
mark=at position 1 with {\arrow{stealth}};
}}
] (0,0) ellipse (2 and 1.5 );

\begin{scope}[C2]
\node at (2,0) [right] {$a_1$};
\node at (-2,0) [left] {$f_1$};
\node at (0,1)  {$b_1$};
\node at (0,-1)  {$g_1$};
\end{scope}

\draw[C1,name path = C1,
postaction={decorate,decoration={markings,
mark=at position .25 with {\arrow{stealth}};,
 mark=at position .25 with {\arrow{stealth}};,
mark=at position .75 with {\arrow{stealth}};,
mark=at position 1 with {\arrow{stealth}};
}}] 
(0,0) ellipse (1.5 and 2 );

\begin{scope}[C1]
\node at (0,2) [above] {$a_2$};
\node at (0,-2) [below] {$f_2$};
\node at (1,0)  {$g_2$};
\node at (-1,0)  {$b_2$};
\end{scope}

\fill [name intersections={of= C1 and C2, name=Q}, delta angle = 90]
[Qs]
\foreach \s in {1,...,4}
{(Q-\s) circle (2pt) ++(-45+\s*90:.4) node  {}};

\begin{scope}
\node at (Q-1) [above right] {$Q_2$};
\node at (Q-2) [above left] {$Q_3$};
\node at (Q-3) [below left] {$Q_4$};
\node at (Q-4) [below right] {$Q_1$};
\end{scope}

 \fill [basepoint] (Q-1) circle (3pt);
\node at (0,-3) {$g_1f_1b_1a_1=1$, $g_2f_2b_2a_2=1$};
\draw[thick, ->] (3.5,0) to node[above]{$\pi$} ++(1.5,0);
\node at (2.5,3) {$\bar D^{(1)}$};
\end{scope}
 
\begin{scope}[ xshift = 7.5cm,every loop/.style={looseness=40, min distance=80}]
\path (0,0) coordinate(P) ;
 \draw[C12] (0,0) to[in=-75, out =-15,loop]node[below right] {$F$} ();
 \draw[C12] (0,0) to[in=15, out =75,loop] node[above right]{$B$} ();
 \draw[C12] (0,0) to[out=165, in =105,loop] node[above left] {$A$} ();
 \draw[C12] (0,0) to[out=-105, in =-165,loop] node[below left]{$G$}();
\draw[fill=black, radius = 2.5pt]   circle;
\node at (2.5,3) {$D^{(1)}$};
 \fill [basepoint] (P) circle (3pt);
\node at (0,-3) {$GFBA=1$};
\end{scope}
\end{tikzpicture}
\end{center}
So the fundamental group of $\bar D$ is the free group on generators $ \inverse b_1a_2$, $g_2\inverse a_1$, and $g_2g_1b_2b_1$. Thus 
\[\pi_1(X)  =  \langle A, B, G \mid  \inverse B A, G \inverse A, G^2B^2 \rangle\isom \IZ/4,\]
because the first two relations give $A=B=G$ in $\pi_1(X)$.

\item[$X_{1.4}$] In this degenerate case of \caseP1 the conic becomes two lines and we have four components $\bar D_1 = L_{12}$, $\bar D_3= L_{34}$, $\bar D_2 = L_{23}$, $\bar D_4 = L_{14}$, where the involution is induced by the permutation $(1234)$. The newly developed nodes will be denoted by $R_1=\bar D_1\cap \bar D_3$ and $R_2 = \bar D_2\cap \bar D_4$.
Thus the map on $1$-skeletons is
\begin{center}
 \begin{tikzpicture}[
scale = .8]
\path
(0,0)++(90:1) coordinate(X1) ++(120:1.73205080757) coordinate (X1+) 
(X1) ++(120-60:1.73205080757) coordinate (X1-)
(0,0)++(210:1)  coordinate(X2) ++(-120:1.73205080757) coordinate (X2+)
(X2) ++(-120-60:1.73205080757) coordinate (X2-)
(0,0)++(330:1) coordinate(X3) ++(0:1.73205080757) coordinate (X3+)
(X3) ++(-60:1.73205080757) coordinate (X3-);

\begin{scope}
\draw[L1] (X2-)  arc  (120+70.89:70.89+240:2.64575131) (250:2.64575131) node[below]{$a_1$};
\draw[L1] (X3-)  arc  (70.89-120:70.89:2.64575131) (10:2.64575131) node[right]{$b_1$};

\draw[L2] (X3) to node[below] {$a_2$} (X2);
\draw[L2] (X2) to node[below] {$b_2$} (X2-);

\draw[L3] (X3-) to node[right] {$f_3$} (X3);
\draw[L3] (X3) to node[right] {$g_3$} (X1);

\draw[L4] (X2) to node[left] {$f_4$} (X1);
\draw[L4] (X1) to node[left] {$g_4$} (X1-);

\foreach \P in { X1, X2, X3, X1-, X2-, X3-} {\fill[Qs] (\P) circle (2pt);};

\begin{scope}
\node at (X1) [left] {$Q_4$};
\node at (X1-) [above right] {$Q_1$};
\node at (X3-) [below right] {$R_1$};
\node at (X2-) [left] {$Q_2$};
\node at (X2) [below]{$R_2$};
\node at (X3) [above right] {$Q_3$};
\end{scope}

\draw[thick, ->] (3.5,0) to node[above]{$\pi$} ++(1.5,0);
\node at  (3,3) {$\bar D^{(1)}$};
\fill[basepoint] (X3) circle (3pt);
\end{scope}

\begin{scope}[ xshift = 7.5cm,every loop/.style={min distance = 80, looseness = 40}]
\coordinate (P') at (0,2);
\coordinate (P) at (0,-1);
 \draw[L12] (P') to[in=180, out =180] node[left] {$B$} (P);
  \draw[L12] (P) to[in=0, out =0] node[right] {$A$} (P');

  \draw[L34] (P') to node[left] {$F$} (P);
  \draw[L34] (P) to[in=-135, out =-45,loop] node[below ] {$G$} ();

  \begin{scope}
\node at (P) [ below] {$Q$};
\node at(P')[ above] {$R$};
\end{scope}

\node at (2.5,3) {$D^{(1)}$};
\fill[basepoint] (P) circle (3pt);
\fill[Qs] (P') circle (2pt);
\end{scope}

\end{tikzpicture}
\end{center}
Again, $\pi_1(\bar D)$ is free, generated by $\inverse g_3f_4a_2$, $f_3\inverse b_1g_4 g_3$, and $f_3a_1b_2a_2$, so that 
\begin{align*}
 \pi_1(X)& \isom \langle  BA, FA, G \mid  \inverse G (F A), F\inverse B G^2, (FA)(BA) \rangle\\
  & \isom \langle BA, FA, G \mid  \inverse G (F A), (FA)\inverse{(BA)}G^2, (FA)(BA) \rangle\\
  & \isom \langle G \mid   G^4 \rangle\\
  & \isom \IZ/4.
\end{align*}

\item[general in \caseP{2}]
Recall that in this case the involution $\tau$ fixes the conic $C$. We can identify the conic with $\IC\cup\{\infty\}$ in such a way  that the intersection points are $\pm1, \pm \I$, which divide  the unit circle in four arcs $a, b, a', b'$. In $C/\tau$ the unit circle is divided in two arcs $A$ and $B$ and clearly $BA=1$ in $\pi_1(C/\tau)$. Adding in the two lines the map on  $1$-skeletons and the  relations are as follows: 
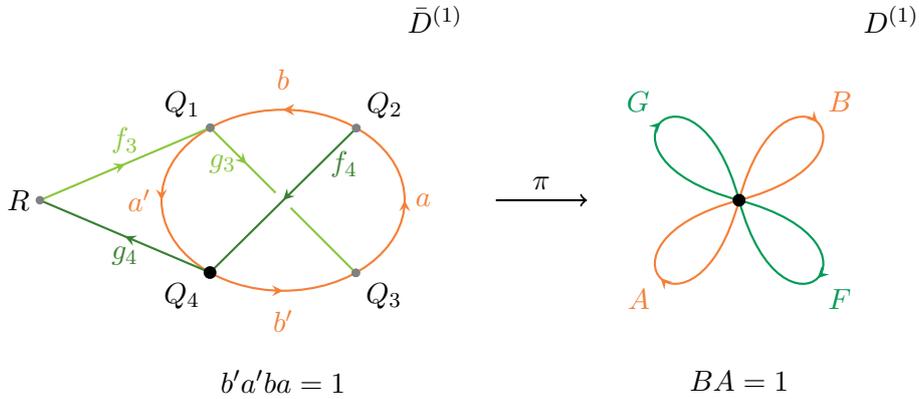
\begin{center}
\begin{tikzpicture}[scale = .8]

\begin{scope}
\draw[C2, name path = C2, 
postaction={decorate,decoration={markings,
mark=at position .25 with {\arrow{stealth}};,
 mark=at position .25 with {\arrow{stealth}};,
mark=at position .75 with {\arrow{stealth}};,
mark=at position 1 with {\arrow{stealth}};
}}
] (0,0) ellipse (2 and 1.5 );

\begin{scope}[C2]
\node at (2,0) [right] {$a$};
\node at (-2,0) [left] {$a'$};
\node at (0,2)  {$b$};
\node at (0,-2)  {$b'$};
\end{scope}

\path[name path = C1] (0,0) ellipse (1.5 and 2 );

\fill [name intersections={of= C1 and C2, name=Q}, delta angle = 90][Qs]
\foreach \s in {1,...,4}
{(Q-\s) circle (2pt) ++(-45+\s*90:.4) node  {}};

\begin{scope}
\node at (Q-1) [above right] {$Q_2$};
\node at (Q-2) [above left] {$Q_1$};
\node at (Q-3) [below left] {$Q_4$};
\node at (Q-4) [below right] {$Q_3$};
\end{scope}

\coordinate (P) at (-4,0);

\begin{scope}
\node at (P) [ left] {$R$};
\end{scope}

\draw[L3] (P) to node[above] {$f_3$} (Q-2);
\draw[L3] (Q-2) to node[left] {$g_3$} (0,0); \draw[L3col, thick](0,0) -- (Q-4);

\fill[white] (0,0) circle (5pt);
\draw[L4] (Q-1) to node[pos= .25, right] {$f_4$} (Q-3);
\draw[L4] (Q-3) to node[below] {$g_4$}(P);

\fill[Qs]
\foreach \s in {1,...,4} {(Q-\s) circle (2pt)}
(P) circle (2pt);

 \fill [basepoint] (Q-3) circle (3pt);
\node at (0,-3) {$b'a'ba=1$};
\draw[thick, ->] (3.5,0) to node[above]{$\pi$} ++(1.5,0);
\node at (2.5,3) {$\bar D^{(1)}$};
\end{scope}
 
\begin{scope}[ xshift = 7.5cm,every loop/.style={looseness=40, min distance=80}]
\path (0,0) coordinate(P) ;
 \draw[L34] (0,0) to[in=-75, out =-15,loop]node[below right] {$F$} ();
 \draw[CC] (0,0) to[in=15, out =75,loop] node[above right]{$B$} ();
 \draw[L34] (0,0) to[out=165, in =105,loop] node[above left] {$G$} ();
 \draw[CC] (0,0) to[out=-105, in =-165,loop] node[below left]{$A$}();
\draw[fill=black, radius = 2.5pt]   circle;
\node at (2.5,3) {$D^{(1)}$};
 \fill [basepoint] (P) circle (3pt);
\node at (0,-3) {$BA=1$};
\end{scope}
\end{tikzpicture}
\end{center}

Again, $\pi_1(\bar D)$ is free, generated by $a' f_3 g_4 $, $ f_4 a b'$, and $ a'\inverse g_3 b'$, so that 
\begin{align*}
 \pi_1(X)& \isom \langle A, B, F, G \mid BA, AFG, FAB, A\inverse GB    \rangle\\
  & \isom \langle A,F, G \mid  AFG, F ,A\inverse G\inverse A \rangle\\
  & \isom \{1\}.
\end{align*}
  \item[$X_{1.1}$] 
   In this case the two components $\bar D_3 = L_{13}$, $\bar D_4= L_{24}$ are exchanged by $\tau$ as described above and the two lines $\bar D_1 = L_{14}$, $\bar D_2 = L_{23}$ are exchanged via the involution induced by $(13)(24)$. The new node $\bar D_1\cap \bar D_2$ will be denoted by $S$.   Then the map on  $1$-skeletons can be chosen to be
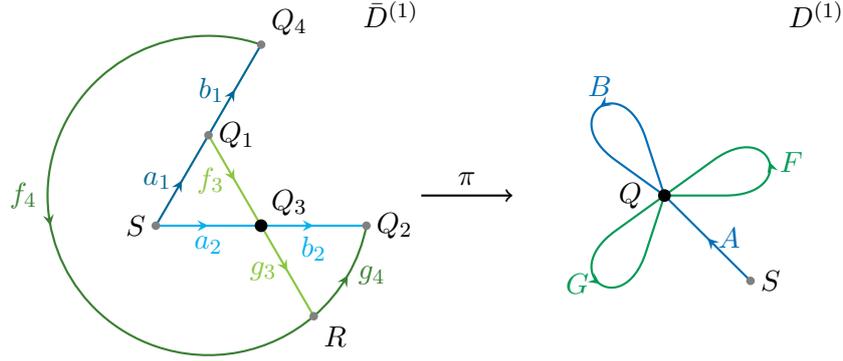
\begin{center}
 \begin{tikzpicture}[scale = .8]
\path
(0,0)++(90:1) coordinate(X1) ++(120:1.73205080757) coordinate (X1+) 
(X1) ++(120-60:1.73205080757) coordinate (X1-)
(0,0)++(210:1)  coordinate(X2) ++(-120:1.73205080757) coordinate (X2+)
(X2) ++(-120-60:1.73205080757) coordinate (X2-)
(0,0)++(330:1) coordinate(X3) ++(0:1.73205080757) coordinate (X3+)
(X3) ++(-60:1.73205080757) coordinate (X3-);

\begin{scope}
\draw[L1] (X2) to node[left] {$a_1$} (X1);
\draw[L1] (X1) to node[left] {$b_1$} (X1-);

\draw[L2] (X2) to node[below] {$a_2$} (X3);
\draw[L2] (X3) to node[below] {$b_2$} (X3+);

\draw[L3] (X1) to node[left] {$f_3$} (X3);
\draw[L3] (X3) to node[left] {$g_3$} (X3-);

\draw[L4] (X1-)  arc  (70.89:310.89:2.64575131) (180:2.64575131) node[left]{$f_4$};
\draw[L4] (X3-)  arc  (310.89:310.89+38.21:2.64575131) (-30:2.64575131) node[right]{$g_4$};

\foreach \P in { X1, X2, X3, X1-, X3+, X3-} {\fill[Qs] (\P) circle (2pt);};

\begin{scope}
\node at (X1) [right] {$Q_1$};
\node at (X1-) [above right] {$Q_4$};
\node at (X3-) [below right] {$R$};
\node at (X3+) [right] {$Q_2$};
\node at (X2) [left] {$S$};
\node at (X3) [above right] {$Q_3$};
\end{scope}

\draw[thick, ->] (3.5,0) to node[above]{$\pi$} ++(1.5,0);
\node at  (3,3) {$\bar D^{(1)}$};
\fill[basepoint] (X3) circle (3pt);
\end{scope}

\begin{scope}[xshift = 7.5cm,looseness=5]
\coordinate (P) at (0,0);
\coordinate (P') at (-45:2);

 \draw[L34] {(P)  -- ++(0:1) to[out=0, in =36]  node [right] {$F$} (36:1) -- (P)};
 \draw[L34] {(P)  -- ++(6*36:1) to[out=6*36, in =7*36]  node [left] {$G$} (7*36:1) -- (P)};

\draw[L12]
\foreach \x in {3*36} {(P)  -- ++(\x:1) to[out=\x, in =\x+36] node[above]{$B$} (\x+36:1) -- (P)};
\draw[L12] (P') -- node[right] {$A$} (P);

\fill[basepoint] (P) circle (3pt);

\begin{scope}
\node at (-.2,0) [ left] {$Q$};
\node at(P')[ right] {$S$};
\end{scope}

\fill [Qs] (P') circle (2pt);
\node at (2.5,3) {$D^{(1)}$};
\end{scope}
\end{tikzpicture}
\end{center}
Again, $\pi_1(\bar D)$ is free, generated by $ f_3 a_1\inverse a_2 $, $\inverse b_2 g_4g_3$, and $\inverse g_3 f_4 b_1 a_1\inverse a_2$, so that 
\begin{align*}
 \pi_1(X)& \isom \langle  B, F, G \mid  F, \inverse B G^2, \inverse GFB \rangle\\
  & \isom \langle B, G \mid  \inverse B G^2, \inverse G B\rangle\\
  & \isom \{1\}.
\end{align*}
\item[$X_{1.2}$]
Again the two components $\bar D_3 = L_{13}$, $\bar D_4= L_{24}$ are exchanged by $\tau$ as described above. The two lines $\bar D_1 = L_{12}$, $\bar D_2 = L_{34}$ are exchanged via the involution induced by $(13)(24)$. The new node $\bar D_1\cap \bar D_2$ will be denoted by $S$.
Then the map on  $1$-skeletons can be chosen to be
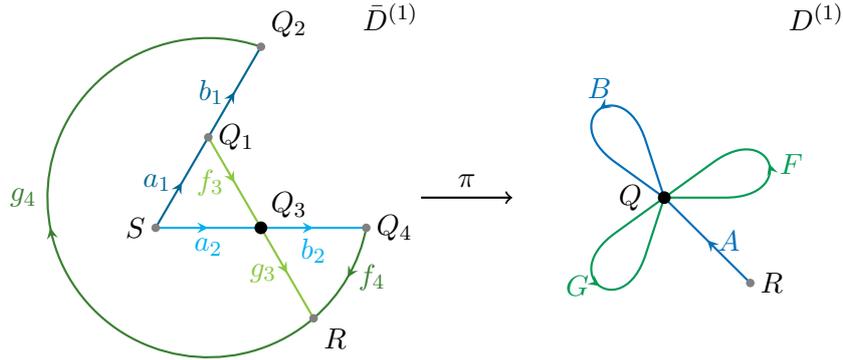
\begin{center}
 \begin{tikzpicture}[scale = .8]
\path
(0,0)++(90:1) coordinate(X1) ++(120:1.73205080757) coordinate (X1+) 
(X1) ++(120-60:1.73205080757) coordinate (X1-)
(0,0)++(210:1)  coordinate(X2) ++(-120:1.73205080757) coordinate (X2+)
(X2) ++(-120-60:1.73205080757) coordinate (X2-)
(0,0)++(330:1) coordinate(X3) ++(0:1.73205080757) coordinate (X3+)
(X3) ++(-60:1.73205080757) coordinate (X3-);

\begin{scope}

\draw[L1] (X2) to node[left] {$a_1$} (X1);
\draw[L1] (X1) to node[left] {$b_1$} (X1-);

\draw[L2] (X2) to node[below] {$a_2$} (X3);
\draw[L2] (X3) to node[below] {$b_2$} (X3+);

\draw[L3] (X1) to node[left] {$f_3$} (X3);
\draw[L3] (X3) to node[left] {$g_3$} (X3-);

\draw[L4] (X3-)  arc  (310.89:70.89:2.64575131) (180:2.64575131) node[left]{$g_4$};
\draw[L4] (X3+)  arc  (310.89+38.21:310.89:2.64575131) (-30:2.64575131) node[right]{$f_4$};

\foreach \P in { X1, X2, X3, X1-, X3+, X3-} {\fill[Qs] (\P) circle (2pt);};

\begin{scope}
\node at (X1) [right] {$Q_1$};
\node at (X1-) [above right] {$Q_2$};
\node at (X3-) [below right] {$R$};
\node at (X3+) [right] {$Q_4$};
\node at (X2) [left] {$S$};
\node at (X3) [above right] {$Q_3$};
\end{scope}

\draw[thick, ->] (3.5,0) to node[above]{$\pi$} ++(1.5,0);
\node at  (3,3) {$\bar D^{(1)}$};
\fill[basepoint] (X3) circle (3pt);
\end{scope}

\begin{scope}[xshift = 7.5cm,looseness=5]
\coordinate (P) at (0,0);
\coordinate (P') at (-45:2);

 \draw[L34] {(P)  -- ++(0:1) to[out=0, in =36]  node [right] {$F$} (36:1) -- (P)};
 \draw[L34] {(P)  -- ++(6*36:1) to[out=6*36, in =7*36]  node [left] {$G$} (7*36:1) -- (P)};

\draw[L12]
\foreach \x in {3*36} {(P)  -- ++(\x:1) to[out=\x, in =\x+36] node[above]{$B$} (\x+36:1) -- (P)};
\draw[L12] (P') -- node[right] {$A$} (P);

\begin{scope}
\node at (-.2,0) [ left] {$Q$};
\node at(P')[ right] {$R$};
\end{scope}

\fill[basepoint] (P) circle (3pt);
\fill [Qs] (P') circle (2pt);
 
\node at (2.5,3) {$D^{(1)}$};
\end{scope}
\end{tikzpicture}
\end{center}

Again, $\pi_1(\bar D)$ is free, generated by $ f_3 a_1\inverse a_2 $, $\inverse b_2 \inverse f_4g_3$, and $\inverse g_3 \inverse g_4 b_1 a_1\inverse a_2$, so that 
\begin{align*}
 \pi_1(X)& \isom \langle  B, F, G \mid  F, \inverse B\inverse F G,  G^{-2}B \rangle\\
  & \isom \langle B, G \mid \inverse B\inverse  G,  G^{-2}B \rangle\\
  & \isom \{1\}.
\end{align*}

\item[general in \caseP{3}] The map on $1$-skeletons and the extra relations are as follows: 
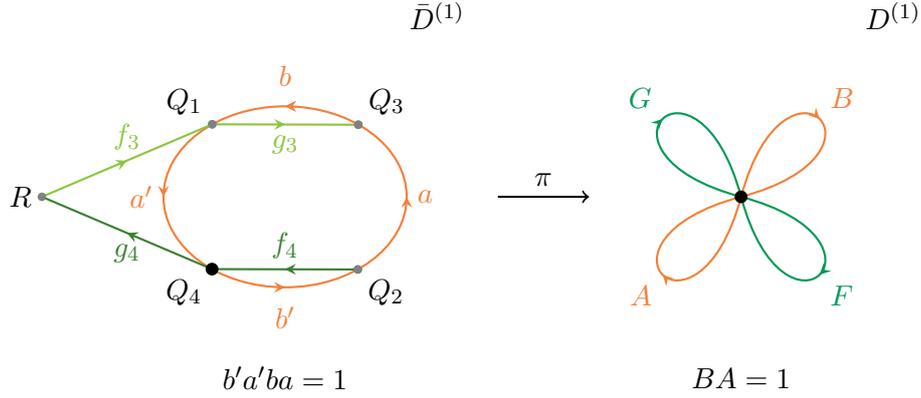
\begin{center}
\begin{tikzpicture}[scale = .8]

\begin{scope}
\draw[C2, name path = C2, 
postaction={decorate,decoration={markings,
mark=at position .25 with {\arrow{stealth}};,
 mark=at position .25 with {\arrow{stealth}};,
mark=at position .75 with {\arrow{stealth}};,
mark=at position 1 with {\arrow{stealth}};
}}
] (0,0) ellipse (2 and 1.5 );

\begin{scope}[C2]
\node at (2,0) [right] {$a$};
\node at (-2,0) [left] {$a'$};
\node at (0,2)  {$b$};
\node at (0,-2)  {$b'$};
\end{scope}

\path[name path = C1] (0,0) ellipse (1.5 and 2 );

\fill [name intersections={of= C1 and C2, name=Q}, delta angle = 90]
[Qs]
\foreach \s in {1,...,4}
{(Q-\s) circle (2pt) ++(-45+\s*90:.4) node  {}};

\begin{scope}
\node at (Q-1) [above right] {$Q_3$};
\node at (Q-2) [above left] {$Q_1$};
\node at (Q-3) [below left] {$Q_4$};
\node at (Q-4) [below right] {$Q_2$};
\end{scope}

\coordinate (P) at (-4,0);
\begin{scope}
\node at (P) [left] {$R$};
\end{scope}

\draw[L3] (P) to node[above] {$f_3$} (Q-2);
\draw[L3] (Q-2) to node[below] {$g_3$} (Q-1);

\draw[L4] (Q-4) to node[above] {$f_4$} (Q-3);
\draw[L4] (Q-3) to node[below] {$g_4$}(P);
\fill[Qs] \foreach \s in {1,...,4} {(Q-\s) circle (2pt)} (P) circle (2pt);

 \fill [basepoint] (Q-3) circle (3pt);

\node at (0,-3) {$b'a'ba=1$};
\draw[thick, ->] (3.5,0) to node[above]{$\pi$} ++(1.5,0);
\node at (2.5,3) {$\bar D^{(1)}$};
\end{scope}
 
\begin{scope}[ xshift = 7.5cm,every loop/.style={looseness=40, min distance=80}]
\path (0,0) coordinate(P);
 \draw[L34] (0,0) to[in=-75, out =-15,loop]node[below right] {$F$} ();
 \draw[CC] (0,0) to[in=15, out =75,loop] node[above right]{$B$} ();
 \draw[L34] (0,0) to[out=165, in =105,loop] node[above left] {$G$} ();
 \draw[CC] (0,0) to[out=-105, in =-165,loop] node[below left]{$A$}();
 \fill[basepoint] (P) circle (3pt);
\node at (2.5,3) {$D^{(1)}$};
\node at (0,-3) {$BA=1$};
\end{scope}
\end{tikzpicture}
\end{center}
Again, $\pi_1(\bar D)$ is free, generated by $a' f_3 g_4 $, $ f_4 b'$, and $ a'\inverse g_3 ab'$, so that 
\begin{align*}
 \pi_1(X)& \isom \langle A, B, F, G \mid BA, AFG, FB, A\inverse GAB \rangle\\
  & \isom \langle A,F, G \mid  AFG, F\inverse A,  A\inverse G\rangle\\
  & \isom \langle A \mid  A^3\rangle\isom \IZ/3.
\end{align*}
\item[$X_{1.3}$]   In this degenerate case of \caseP3 the two components $\bar D_3 = L_{13}$, $\bar D_4= L_{24}$ are exchanged by $\tau$ as described above and the two lines $\bar D_1 = L_{14}$, $\bar D_2 = L_{23}$ are exchanged via the involution induced by $(12)(34)$. The new node $\bar D_1\cap \bar D_2$ will be denoted by $S$. Then the map on $1$-skeletons is
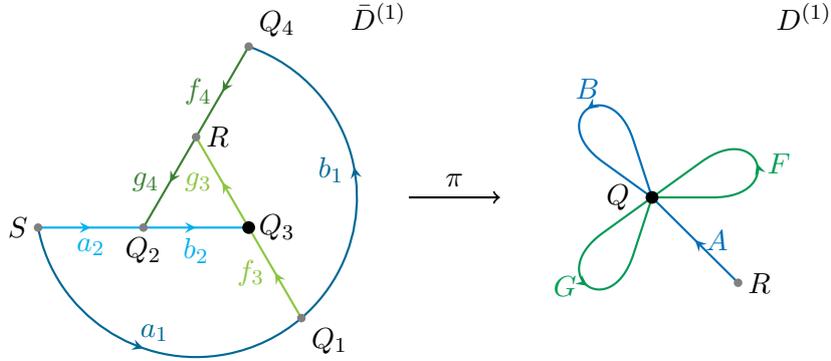
\begin{center}
 \begin{tikzpicture}[scale = .8]
\path
(0,0)++(90:1) coordinate(X1) ++(120:1.73205080757) coordinate (X1+) 
(X1) ++(120-60:1.73205080757) coordinate (X1-)
(0,0)++(210:1)  coordinate(X2) ++(-120:1.73205080757) coordinate (X2+)
(X2) ++(-120-60:1.73205080757) coordinate (X2-)
(0,0)++(330:1) coordinate(X3) ++(0:1.73205080757) coordinate (X3+)
(X3) ++(-60:1.73205080757) coordinate (X3-);

\begin{scope}
\draw[L1] (X2-)  arc  (120+70.89:70.89+240:2.64575131) (255:2.64575131) node[above ]{$a_1$};
\draw[L1] (X3-)  arc  (70.89-120:70.89:2.64575131) (10:2.64575131) node[left]{$b_1$};

\draw[L2] (X2-) to node[below] {$a_2$} (X2);
\draw[L2] (X2) to node[below] {$b_2$} (X3);

\draw[L3] (X3-) to node[left] {$f_3$} (X3);
\draw[L3] (X3) to node[left] {$g_3$} (X1);

\draw[L4] (X1-) to node[left] {$f_4$} (X1);
\draw[L4] (X1) to node[left] {$g_4$} (X2);

\foreach \P in { X1, X2, X3, X1-, X2-, X3-} {\fill[Qs] (\P) circle (2pt);};

\begin{scope}
\node at (X1) [right] {$R$};
\node at (X1-) [above right] {$Q_4$};
\node at (X3-) [below right] {$Q_1$};
\node at (X2-) [left] {$S$};
\node at (X2) [below]{$Q_2$};
\node at (X3) [right] {$Q_3$};
\end{scope}

\draw[thick, ->] (3.5,0) to node[above]{$\pi$} ++(1.5,0);
\node at  (3,3) {$\bar D^{(1)}$};
 \fill[basepoint] (X3) circle (3pt);
\end{scope}

\begin{scope}[xshift = 7.5cm,looseness=5]
\coordinate (P) at (0,0);
\coordinate (P') at (-45:2);
 \draw[L34] {(P)  -- ++(0:1) to[out=0, in =36]  node [right] {$F$} (36:1) -- (P)};
 \draw[L34] {(P)  -- ++(6*36:1) to[out=6*36, in =7*36]  node [left] {$G$} (7*36:1) -- (P)};

\draw[L12]
\foreach \x in {3*36} {(P)  -- ++(\x:1) to[out=\x, in =\x+36] node[above]{$B$} (\x+36:1) -- (P)};

\begin{scope}
\node at (-.2,0) [ left] {$Q$};
\node at(P')[ right] {$R$};
\end{scope}

\draw[L12] (P') -- node[right] {$A$} (P);
 \fill[basepoint] (P) circle (3pt);
 \fill [Qs] (P') circle (2pt);
 
\node at (2.5,3) {$D^{(1)}$};
\end{scope}
\end{tikzpicture}
\end{center}
Again, $\pi_1(\bar D)$ is free, generated by $b_2g_4g_3$, $\inverse g_3 f_4b_1\inverse f_3$, and $b_2 a_2\inverse a_1 \inverse f_3$, so that 
\begin{align*}
 \pi_1(X)& \isom \langle  B, F, G \mid  BGG,\inverse G F B \inverse F,  B\inverse F \rangle\\
  & \isom \langle G\mid  G^3 \rangle\\
  & \isom \IZ/3.
\end{align*}

\item[$X_{1.5}$] This case does not come as the degeneration of the previous ones and we deviate slightly from the notation: the surface $X_{1.5}$ has two degenerate cusps at points $Q$ and $R$ and we mark the corresponding preimages with $Q_i$ respectively $R_i$. The map on $1$-skeletons can be given as
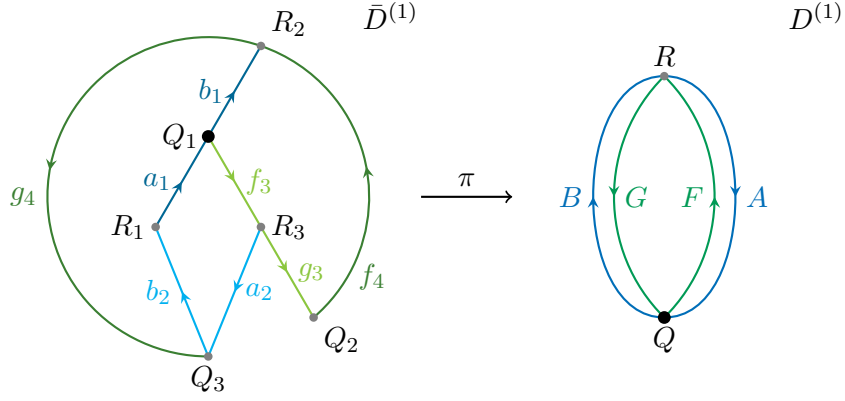
\begin{center}
 \begin{tikzpicture}[
scale = .8]
\path
(0,0)++(90:1) coordinate(X1) ++(120:1.73205080757) coordinate (X1+) 
(X1) ++(120-60:1.73205080757) coordinate (X1-)
(0,0)++(210:1)  coordinate(X2) ++(-120:1.73205080757) coordinate (X2+)
(X2) ++(-120-60:1.73205080757) coordinate (X2-)
(0,0)++(330:1) coordinate(X3) ++(0:1.73205080757) coordinate (X3+)
(X3) ++(-60:1.73205080757) coordinate (X3-);
\coordinate (X4) at (270:2.64575131);

\begin{scope}

\draw[L1] (X2) to node[left] {$a_1$} (X1);
\draw[L1] (X1) to node[left] {$b_1$} (X1-);

\draw[L2] (X3) to node[right] {$a_2$} (X4);
\draw[L2] (X4) to node[left] {$b_2$} (X2);

\draw[L3] (X1) to node[right] {$f_3$} (X3);
\draw[L3] (X3) to node[right] {$g_3$} (X3-);

\draw[L4] (X3-)  arc  (310.89:310.89+120:2.64575131) (-30:2.64575131) node[right]{$f_4$};
\draw[L4] (X1-)  arc  (70.89:270:2.64575131) (180:2.64575131) node[left]{$g_4$};

\foreach \P in { X1, X2, X3, X1-, X4, X3-} {\fill[Qs] (\P) circle (2pt);};
\draw[thick, ->] (3.5,0) to node[above]{$\pi$} ++(1.5,0);
\fill[basepoint] (X1) circle (3pt);
\node at  (3,3) {$\bar D^{(1)}$};

\begin{scope}
\node at (X1) [left] {$Q_1$};
\node at (X4) [below] {$Q_3$};
\node at (X3-) [below right] {$Q_2$};
\node at (X1-) [above right] {$R_2$};
\node at (X2) [left]{$R_1$};
\node at (X3) [right] {$R_3$};
\end{scope}

\end{scope}

\begin{scope}[ xshift = 7.5cm]
\coordinate (P) at (0,2);
\coordinate (P') at (0,-2);
 \draw[L12] (P') to[in=180, out =180] node[left] {$B$} (P);
  \draw[L12] (P) to[in=0, out =0] node[right] {$A$} (P');

  \draw[L34] (P') to[in=-45, out =45] node[left] {$F$} (P);
  \draw[L34] (P) to[in=135, out =-135] node[right] {$G$} (P');

\node at (2.5,3) {$D^{(1)}$};
\fill[basepoint] (P') circle (3pt);
\fill[Qs] (P) circle (2pt);
\node at (P')[below] {$Q$};
\node at (P) [above] {$R$};
\end{scope}
\end{tikzpicture}
\end{center}
Again, $\pi_1(\bar D)$ is free, generated by $a_1b_2g_4b_1$, $\inverse b_1f_4g_3f_3$, and $a_1b_2a_2f_3$, so that 
\begin{align*}
 \pi_1(X)& \isom \langle  AB, AF, GF \mid  ABGB, \inverse BFGF, ABAF \rangle\\
& \isom \langle  AB, AF, GF \mid  (AB)(GF)\inverse{(AF)}(AB), \inverse {(AB)}(A F)(GF), (AB)(AF) \rangle\\
& \isom \langle  AB, AF \mid  (AB)(GF)(AB)^2, (AB)^{-2}(GF)\rangle\\
& \isom \langle  AB \mid  (AB)^5\rangle\\
  & \isom \IZ/5.
  \end{align*}
Indeed, the universal cover of $X_{1.5}$ is a $\IZ/5$-invariant quintic in $\IP^3$, which is the union of $5$ planes (compare \cite[Rem.~9]{franciosi-rollenske16}). 
\end{description}
\end{proof}

\begin{prop}\label{prop: pi1 case E+}
 Let $X$ be a Gorenstein stable   Godeaux surface. 
 \begin{enumerate}
 \item 
If $X$ is of type \caseE{m} with irreducible polarisation  then $\pi_1(X)\isom \IZ/m$.
\item 
If $X$ is of type \caseE{m} with reducible polarisation then $X$ is simply connected.
\end{enumerate}
 \end{prop}
 
\begin{proof}
\refenum{i} We compute the fundamental group case by case.

Consider the situation of Example \ref{ex: odd} with $\bar D$ smooth:
 \[\begin{tikzcd}
    F\dar \arrow[hookrightarrow]{r}{\phi \times\phi' } & D \times D' \dar\\
   \bar F = F/F[2] \rar & \frac{D \times D' }{F[2]}=A \rar{q}\dar{\pi } & A/\bar F\\
   & D /D [2]
   \end{tikzcd}
 \]
Let $d=\deg \phi$, $d' = \deg \phi'$ and recall that $d +d'  = 6$ and both numbers are odd. For all the elliptic curves  we choose lattices in such a way   that all morphisms in the diagram are induced by the identity on $\IC^2$. More specifically, we choose a non-real complex number $\tau$ such that as sublattices of $\IC$ we describe the integral first homology groups as
\begin{gather*}
 H_1(F) = \langle 2d, 2d'\tau \rangle \subset H_1(\bar F) = \langle d, d'\tau \rangle,\\
 H_1(D) = \langle 2, 2d'\tau\rangle\subset H_1(D/D[2]) = \langle1 , d'\tau\rangle,\\
 H_1(D') = \langle 2d, 2\tau\rangle.
\end{gather*}
This is obvious if $\{d, d'\}=\{1,5\}$; if $d=d'=3$ then this is easy to achieve, since $F[3] = \ker\phi_1\oplus \ker\phi_2$.
With these choices we have 
\begin{equation}\label{eq: H_1A}
H_1(A) =\langle (2, 0), (2d'\tau, 0 ) , (0, 2d),(0, 2\tau), ( d, d), (d'\tau, d'\tau)\rangle \subset \IC^2.
\end{equation}
We denote by $a=1$ and $b=d'\tau$ the generators of $H_1(D/D[2])$ and by $\alpha$ and $\beta$ generators of $H_1(A/\bar F)$ (to be specified below).

Recall from Proposition \ref{prop: Godeaux from E+ vs. bitrielliptic configurations}  that $X$ is the stable surface associated to the triple $(\bar X = S^2 (A/\bar F), \bar D , \sigma)$,  where $\sigma$ is the covering involution for the double cover $\pi\colon \bar D \to D/D[2]$. Thus by \cite[Cor.~3.2\,(iii)]{FPR15b}.
we have 
\begin{equation}\label{eq: p_1X}
 \begin{split}
\pi_1(X) & = H_1(D/D[2])\ast_{H_1(\bar D)} H_1(A/\bar F)\\
&=\left\langle a,b, \alpha, \beta\mid [a,b]=[\alpha,\beta]=0, {\pi}_*\gamma_i = q_*\gamma_i\right\rangle,  
 \end{split}
\end{equation}
where the $\gamma_i$ generate $H_1(\bar D) = H_1(A)$.
\begin{description}
 \item[Case  $d =1$] In this case \eqref{eq: H_1A} becomes
  \[H_1(A) =\langle  (0, 2), (0, 2\tau), \underbrace{( 1, 1) , (5\tau, 5\tau)}_{=H_1(\bar F)}\rangle\]
so that we can choose $\alpha = q_*(0,2)$ and $\beta=q_*(0, 2\tau)$.   Writing out the induced relations in \eqref{eq: p_1X} we get
   \[ \alpha=0, \, \beta=0, \, a = 0,\, b = 0,\]
and thus $\pi_1(X)$ is a trivial group.
 \item[Case  $d =3$] In this case \eqref{eq: H_1A} becomes
  \[H_1(A) =\langle  (2, 0), (0, 2\tau), \underbrace{( 3,3) , (3\tau, 3\tau)}_{=H_1(\bar F)}\rangle\]
so that we can choose $\alpha = q_*(2, 0)$ and $\beta=q_*(0, 2\tau)$.   Writing out the induced relations in \eqref{eq: p_1X} we get
   \[  \alpha=2a , \, \beta=0 , \, 3a = 0,\, b = 0,\]
and thus $\pi_1(X)\isom \IZ/3\IZ$.
 \item[Case  $d =5$] In this case \eqref{eq: H_1A} becomes
  \[H_1(A) =\langle  (2, 0), ( 2\tau, 0), \underbrace{( 5, 5) , (\tau, \tau)}_{=H_1(\bar F)}\rangle\]
so that we can choose $\alpha = q_*(2,0)$ and $\beta=q_*( 2\tau,0)$.   Writing out the induced relations in \eqref{eq: p_1X} we get
   \[  \alpha=2a , \, \beta=2b, \, 5a = 0,\, b = 0,\]
and thus $\pi_1(X)\isom \IZ/5\IZ$.
\end{description}

We now consider the situation of Example \ref{ex: even} with irreducible polarisation:
\[\begin{tikzcd}
    F\dar \arrow[hookrightarrow]{r}{\phi\times\phi'} & D\times D'\dar\\
   \bar F = F/F\cap G \rar & \frac{D\times D'}{G}=A \rar{q}\dar{\pi} & A/\bar F\\
   & D/D[2]
   \end{tikzcd}.
 \]
 Let $d=\deg \phi$, $d' = \deg \phi'$ and recall that $d +d'  =3$. The subgroup $G\isom(\IZ/2\IZ)^2$ is chosen inside $D[2]\times D'[2]$ so  that $G\cap (D\times\{0\})=G\cap (\{0\}\times D')=\{0\}$ and $G\cap  F  = \langle\xi\rangle$ where $\xi$ is a non-zero element in $F[2]$. If we denote by   $\zeta$ the only non-zero element in $\ker\phi+\ker \phi'\subset F[2]$ (one of the maps has degree $1$ and the other one degree $2$), then by construction $F[2]$ is generated by $\xi$ and $\zeta$.
 
 As in the first part of the proof, we choose a convenient representation of the homology lattices. We can identify $F = \IC/\langle 4, 2\tau\rangle$ so that $\tau$ is in the upper half plane, $\xi$ is the image of $\tau$, and $\zeta$ is the image of $2$. Then the other lattices are naturally described as follows:\footnote{To check this,  note that if $d' = 2$ then $H_1(F)\to H_1(D)$ should be the identity.}
\begin{gather*}
 H_1(F) = \langle 4,2\tau \rangle\subset H_1(\bar F) = \langle 4, \tau\rangle,\\
 H_1(D) = \langle 2d',2\tau\rangle\subset H_1(D/D[2]) = \langle d' , \tau\rangle,\\
 H_1(D') = \langle 2d, 2\tau\rangle.
\end{gather*}
A straightforward computation in $D[2]\times D'[2]$ shows that the only possibilities for $G$ are 
\[  G_1 = \langle (\tau, \tau), (d', d)\rangle, \, G_2 = \langle (\tau, \tau), (d'+\tau, d)\rangle.\]

If we denote by $a=d'$ and $b=\tau$ the generators of $H_1(D/D[2]) $ and with $\alpha$ and $\beta$ generators for $H_1(A/\bar F)$ we can compute the fundamental group as in \eqref{eq: p_1X}.
\begin{description}
 \item[Case $d' = 1$, $G=G_1$] In this case 
 \begin{gather*}H_1(A) =\langle  (2, 0),( 2\tau,0), (1 ,2), \underbrace{(4,4), (\tau, \tau)}_{=H_1(\bar F)}\rangle,\\
  2 (1,2)+(2,0)\equiv 0 \mod  H_1(\bar F).
 \end{gather*}
We can choose $\beta = q_*(2\tau,0)$, $\alpha=q_*(1,2)$ 
so that we get  the relations  \eqref{eq: p_1X} 
\[   \alpha=a,  \beta=2b, 4a = 0, b = 0  \]
and thus $\pi_1(X) \isom \IZ/4\IZ$.
 \item[Case $d' = 1$, $G=G_2$] In this case 
 \begin{gather*}H_1(A) =\langle  (2, 0),(0, 2\tau), (1+\tau ,2), \underbrace{(4,4), (\tau, \tau)}_{=H_1(\bar F)}\rangle,\\
  2 (1+\tau,2)+(2,0)+(0,2\tau)\equiv 0 \mod  H_1(\bar F).
 \end{gather*}
We can choose $\alpha = q_*(0,2\tau)$, $\beta=q_*(1+\tau,2)$ 
so that  the relations in \eqref{eq: p_1X} become
\[ \alpha=0,  \beta=a+b,  4a = 0, b = 0  \]
and thus $\pi_1(X) \isom \IZ/4\IZ$.
\item[Case $d'= 2$, $G=G_1$] In this case 
 \begin{gather*}H_1(A) =\langle  (0,2),(0, 2\tau), (2 ,1), \underbrace{(4,4), (\tau, \tau)}_{=H_1(\bar F)}\rangle,\\
  2 (2,1)+(0,2)\equiv 0 \mod  H_1(\bar F).
 \end{gather*}
We can choose $\alpha = q_*(0,2\tau)$ and $\beta=q_*(2,1)$ so that   the relations in \eqref{eq: p_1X} become
\[ \alpha=0, \beta=a, 2a=0,   b = 0  \]
and thus $\pi_1(X) \isom \IZ/2\IZ$.
\item[Case $d' = 2$, $G=G_2$] In this case 
 \begin{gather*}H_1(A) =\langle  (0,2),( 0,2\tau), (2 +\tau,1), \underbrace{(4,4), (\tau, \tau)}_{=H_1(\bar F)}\rangle,\\
  2 (2+\tau,1)+(0,2)+(0,2\tau)\equiv 0 \mod  H_1(\bar F).
 \end{gather*}
We can choose $\alpha = q_*(0,2\tau)$ and $\beta=q_*(2+\tau,1)$, so that 
the relations in \eqref{eq: p_1X} become
\[\alpha=0,  \beta=a+b,  2a = 0, b = 0  \]
and thus $\pi_1(X) \isom \IZ/2\IZ$.
\end{description}

\refenum{ii}
Now assume that $E$ is an elliptic curve  and  $\bar D\subset S^2E$  is a reducible bi-tri-elliptic curve of genus 2. Then as explained in Remark \ref{rem: reducible bi-tri-elliptic} we have $\bar D = C_1+C_2$ where $C_1\isom E$ is a section and $C_2$ is a bisection which meet in one point.  Furthermore, the bi-elliptic map  $\bar D \to D=C_1$ is   induced by the identity on $C_1$ and  restricts to an isomorphism $\phi\colon C_2\to C_1$. 

Thus the subgroup $\pi_1(C_1)\ast \{1\} \subset \pi_1(\bar D) = \pi_1(C_1) \ast \pi_1(C_2)$ maps isomorphically both onto $\pi_1(\bar X) \isom   \pi_1(E)$ and onto $\pi_1(D) = \pi_1(C_1)$. Hence
\begin{align*}
 \pi_1(X) &\isom  \pi(\bar X) \ast_{\pi_1(\bar D)}\pi_1(D)\\
 &\isom  \pi_1(C_1) \ast_{\left(\pi_1(C_1)\ast\pi_1(C_2)\right)}\pi_1(C_1)\\
 &\isom  \frac{\pi_1(C_1)}{\phi_*\pi_1(C_2)}
\isom \{1\},
\end{align*}
because $\phi_*$ is an isomorphism.
\end{proof}

\subsection{Case \casedP}\label{sect: case dP}
In this case $\bar X$ is a del Pezzo surface of degree $1$ (possibly with canonical singularities) and $\bar D\in |-2K_{\bar X}|$ is a nodal curve. If there exists an involution $\tau$ such that $(\bar X, \bar D, \tau)$ gives rise to a Gorenstein stable Godeaux surface $X$ then \cite[Lem.~3.5]{FPR15a} implies that $\bar D$ is the union of two nodal anti-canonical curves and that the non-normal locus $D\subset X$ has arithmetic genus $2$. Thus there is a unique possibility for the glueing and we get the following:
\begin{prop}\label{prop: Godeaux from dP}
If $X$ is a Gorenstein stable Godeaux surface with normalisation a   del Pezzo surface of degree $1$. Then $X$ is as described in \cite{rollenske16}. Moreover, $X$ is simply connected and not smoothable. 
\end{prop}

\section{Conclusions}\label{sect: conclusions}
In this section we collect the results of the previous sections and in particular prove all claims contained in Table \ref{tab: list}. 

\begin{thm}\label{thm: total}
 Let $X$ be a Gorenstein stable Godeaux surface with non-canonical singularities. 
 \begin{enumerate}
  \item If $X$ is normal then it is either of type \caseR\  or \caseB{1} or \caseB{2} described in Section \ref{sect: normal classification}.
  \item 
 If $X$ is not normal, then $X$ either has normalisation $\IP^2$ and is as described in Proposition \ref{prop: Godeaux from P}, or the normalisation of $X$ is a symmetric product of elliptic curves and $X$ is as in Proposition \ref{prop: classification Godeaux from E+}, or the normalisation of $X$ is a del Pezzo surface of degree $1$ and $X$ is as described in Section \ref{sect: case dP}.
 \end{enumerate}
 \end{thm}
\begin{proof}
 The first item is Proposition \ref{prop: normal-class} with Proposition \ref{prop:godeaux-normal-B}.
 
 The second item follows from \cite[Thm.~3.6]{FPR15a} together with Propositions \ref{prop: Godeaux from P}, \ref{prop: classification Godeaux from E+}, and  \ref{prop: Godeaux from dP}.
\end{proof}

\begin{cor}
 Let $X$ be a Gorenstein stable Godeaux surface. Then $\pi_1^\text{alg}(X)$ is cyclic of order at most $5$. If $X$ has non-canonical singularities or  $|\pi_1^\text{alg}(X)|\geq 3$ then in addition $\pi_1(X)=\pi_1^\text{alg}(X)$.
\end{cor}
\begin{proof}
The case of canonical singularities is in \cite{miyaoka76}, and   the equality between algebraic and topological fundamental group  in the case of large torsion follows from the classification in \cite{reid78}.
 
The cases with non-canonical singularities are listed in Theorem \ref{thm: total} and their fundamental groups have been computed in Theorem \ref{thm:godeaux-normal-torsion}, Propositions \ref{prop: pi1 case P} and \ref{prop: pi1 case E+} and \cite[Prop.~2.3]{rollenske16}. Since in every case these are finite cyclic of order at most $5$, they coincide with the algebraic fundamental groups.
\end{proof}

It is an interesting question, whether the examples that we construct are smoothable, that is, occur as degenerations of smooth Godeaux surfaces. As so far we can only attack this problem by an indirect method, i.\,e., by extending the classification of classical Godeaux surfaces with $|T(X)|\geq 3$ due to Reid \cite{reid78} to Gorenstein stable surfaces. 
\begin{thm}[Theorem 1 in \cite{franciosi-rollenske16}]\label{thm: smoothable T>2}
 Every Gorenstein stable Godeaux surface with $|T(X)|\geq 3$ is smoothable. 
\end{thm}
On the other hand it was shown the simply-connected examples of type \casedP\ are not smoothable \cite{rollenske16}. In other words, the moduli space of stable Godeaux surfaces has at least one irreducible component which does not contain any smooth surface. 

Unlike in the classical case, topological invariants do not distinguish between different connected components of the moduli space, since the topology of degenerations can differ drastically from the topology of a general fibre in the family. However, we can show the following. 

\begin{prop}\label{prop:normal-torsion}
The compactification of the moduli space of Godeaux surfaces with torsion 0 or $\IZ/2$ contains no normal Gorenstein surface with worse than canonical singularities.
\end{prop}
Note  that the Gorenstein assumption is crucial  as is shown by the examples in \cite{lee-park07, urzua16pre}. 
Proposition \ref{prop:normal-torsion} follows directly from Theorem \ref{thm:godeaux-normal-torsion} and  the following general result, for whose proof  we are indebted to Angelo Vistoli.
\begin{prop}\label{prop:angelo}
Let $\Delta\subset \IC$ be the unit disk and let $\pi \colon \mathcal X\to \Delta$ be a proper flat family of over $\Delta$ with reduced connected fibres. \newline Denote by $X_t$ the fibre over $t\in \Delta$; if  
$T(X_0)$ contains a cyclic subgroup of order $d$, then  $T(X_t)$ also contains such a subgroup. 
\end{prop}
\begin{proof}
Let $Y_0\to X_0$ be an \'etale cyclic cover of degree $d$. We are going to show that, possibly after a base change with a map $\Delta\to \Delta$ of the form $t\mapsto t^m$, $Y_0\to X_0$ is induced by  a cyclic \'etale cover $\mathcal Y\to \mathcal X$. This is the same as proving that, up to a base change as above,  the map $H^1(\mathcal X, \IZ/d)\to H^1(X_0,\IZ/d)$ is surjective. 
Consider the Leray spectral sequence $E_2^{p,q}= H^p(\Delta, R^q\pi_*(\IZ/d)_X)\Rightarrow H^{p+q}(\mathcal X,\IZ/d)$. Since $\pi_*(\IZ/d)_X=({\IZ/d})_{\Delta}$ and the cohomology of $({\IZ/d})_{\Delta}$ is trivial, one has an isomorphism $H^1(\mathcal X,\IZ/d)\to H^0(\Delta, R^1\pi_*(\IZ/d)_X)$. So it suffices to show that it is possible to find a base change $\Delta\to \Delta$ such that $H^0(\Delta,R^1\pi_*(\IZ/d)_X)\to H^1(X_0, \IZ/d)$ becomes  surjective.

In order to prove this, we look at the following  exact  sequence:
\begin{equation}\label{eq:seqj}
0\to j_!j^*R^1\pi_*(\IZ/d)_X\to R^1\pi_*(\IZ/d)_X\to H^1(X_0,\IZ/d)_0\to 0,
\end{equation}
where $\Delta^*=\Delta\setminus \{0\}$ and $j\colon \Delta^*\to \Delta$ is the inclusion and $H^1(X_0,\IZ/d)_0$ is the sheaf concentrated in $0$ with stalk $H^1(X_0,\IZ/d)$.
Sequence \ref{eq:seqj} shows that the obstruction to the surjectivity of the map we are interested in lies in $H^1(\Delta,  j_!j^*R^1\pi_*(\IZ/d)_X)$. So it is enough to show that this group vanishes after taking a suitable base change $\Delta\to\Delta$. Indeed, $j^*R^1\pi_*(\IZ/d)_X$  is a locally constant sheaf on $\Delta^*$ and it has finite monodromy since $H^1(X_t, \IZ/d)$ is a finite group for $t\in \Delta$. So we can find a cover $\Delta\to \Delta$ ramified in $0$ such that $j^*R^1\pi_*(\IZ/d)_X$ is    a constant sheaf $A_{\Delta^*}$ associated with a finite abelian group $A$. Consider now the sequence:
$$0\to j_!(A_{\Delta^*})\to A_{\Delta}\to A_0\to 0,$$
where $A_0$ denotes the group $A$ concentrated at $0$; taking cohomology, one immediately gets  $H^1\left(\Delta, j_!(A_{\Delta^*})\right)=0$.
\end{proof}

\begin{rem}\label{rem: smoothing reducible polarisation}
 Note however that the fundamental group need not be constant under degeneration. As an explicit example consider the following: consider a bi-tri-elliptic smoothing of a nodal bi-tri-elliptic curve. Then performing the construction of the surface of type \caseE{m} in families, we see that the one with reduced polarisation occurs as a degeneration of surfaces of type \caseE{m} with irreducible polarisation. By Proposition \ref{prop: pi1 case E+} the former is simply connected while the latter has cyclic fundamental group of order $m\leq 5$, so $\pi_1$ gets smaller in the special fibre.
 
 On the other hand, since Godeaux surfaces with $|T(X)|\geq 3$ are smoothable by \cite[Thm.~A]{franciosi-rollenske16}, this construction implies that all surfaces of type \caseE{m} are smoothable if $m\geq3$.

 We do not dare speculate, whether  it is at these points that  the moduli space of stable Godeaux surfaces becomes connected.
\end{rem}

 \bibliographystyle{../halpha}
 \bibliography{../srollens}

 \end{document}